\documentclass[12pt]{article}
\usepackage{amssymb}
\usepackage{amsmath}

\newtheorem{theorem}{Theorem}

\newtheorem{corollary}[theorem]{Corollary}

\newtheorem{lemma}[theorem]{Lemma}

\newtheorem{proposition}[theorem]{Proposition}
\newtheorem{remark}[theorem]{Remark}

\newenvironment{proof}[1][Proof]{\noindent\textbf{#1.} }{\ \rule{0.5em}{0.5em}}
\input{tcilatex}

\begin{document}

\date{ }
\title{\textbf{Asymptotic behaviour of a cylindrical elastic structure
periodically reinforced along identical fibers}}
\author{Mustapha EL JARROUDI\thanks{%
Email: eljar@fstt.ac.ma} \\
\textit{D\'epartement de Math\'ematiques, Facult\'e des Sciences et
Techniques}\\
\textit{B.P.\ 416, Tanger, MOROCCO}\\
\smallskip and\smallskip  \and Alain BRILLARD\thanks{%
Corresponding author.\ Email : A.Brillard@uha.fr} \\
\textit{Universit\'e de Haute-Alsace, 2\ rue des Fr\`eres Lumi\`ere, }\\
\textit{F-68093\ Mulhouse Cedex, FRANCE}}
\maketitle

\begin{abstract}
We describe the asymptotic behaviour of a cylindrical elastic body,
reinforced along identical $\varepsilon $-periodically distributed fibers of
size $r_\varepsilon $, with $0<r_\varepsilon <\varepsilon $, filled in with
some different elastic material, when this small parameter $\varepsilon $
goes to 0. The case of small deformations and small strains is considered.
We exhibit a critical size of the fibers and a critical link between the
radius of the fibers and the size of the Lam\'e coefficients of the
reinforcing elastic material. Epi-convergence arguments are used in order to
prove this asymptotic behaviour. The proof is essentially based on the
construction of appropriate test-functions.\medskip

\noindent \textbf{Keywords : }Reinforcement, fibers, linear elasticity,
epi-convergence.

\noindent \textbf{AMS Subject Classification (2000) :} 74E30, 74G10, 35B27
\end{abstract}

\section{Introduction}

The purpose of this work is to determine the asymptotic behaviour of an
elastic material periodically reinforced by means of identical fibers filled
in with some isotropic and homogeneous elastic material. In the first part,
the fibers are longitudinally distributed inside the elastic material.\ The
limit law is derived, studying the convergence of the elastic energy, and we
exhibit a critical size of the fibers and a critical size of the Lam\'e
coefficients of the reinforcing fibers. In the last part of this work, we
suppose that the fibers are transversally distributed and we exhibit the
limit law, which still involves a critical size and a critical size of the
Lam\'e coefficients of the fibers, but working in a different limit
functional space. These configurations intend to modelize, for example, the
behaviour of a strap reinforced by means of identical fibers which are
longitudinally or transversally disposed inside the strap.

Let $\omega $ be a bounded, smooth and open subset of $\mathbf{R}^{2}$ and $%
\Omega =\omega \times \left] 0,L\right[ \subset \mathbf{R}^{3}$, where $L$
is positive. $\Gamma _{1}$ denotes the lower basis of $\Omega $ : $\Gamma
_{1}=\omega \times \left\{ 0\right\} $, $\Gamma _{2}$ its upper basis : $%
\Gamma _{2}=\omega \times \left\{ L\right\} $ and $\Sigma $ its lateral
surface : $\Sigma =\partial \omega \times \left] 0,L\right[ $.

Let $\varepsilon $ be some positive real.\ In the first part of this work,
we dispose inside $\Omega $ longitudinal fibers.\ More precisely, for every $%
k=\left( k_{1},k_{2}\right) $ in $\mathbf{Z}^{2}$, we define the square : $%
Y_{\varepsilon }^{k}=\left( \varepsilon k_{1},\varepsilon k_{2}\right) +%
\left] -\varepsilon /2,\varepsilon /2\right[ ^{2}$. Then we denote by $%
Y_{\varepsilon }$ the union of all the $\varepsilon $-cells $Y_{\varepsilon
}^{k}$ included in $\omega $ : $Y_{\varepsilon }=\cup _{k\in K(\varepsilon
)}Y_{\varepsilon }^{k}$. Choosing a parameter $r_{\varepsilon }$ smaller
than $\varepsilon $, we consider the disk $D_{\varepsilon }^{k}$ of radius $%
r_{\varepsilon }$ contained in $Y_{\varepsilon }^{k}$ and the cylinder $%
T_{\varepsilon }^{k}=D_{\varepsilon }^{k}\times \left] 0,L\right[ $. $%
T_{\varepsilon }$ denotes the union $\cup _{k}T_{\varepsilon }^{k}$ of the
cylinders $T_{\varepsilon }^{k}$ contained in $\Omega $. Thus $\overline{%
T_{\varepsilon }}\cap \Sigma $ is empty. The total number of such cylinders
contained in $\Omega $ (that is the cardinal of $K\left( \varepsilon \right) 
$) is equivalent to $\left\vert \omega \right\vert /\varepsilon ^{2}$, with $%
\left\vert \omega \right\vert =area\left( \omega \right) $. The domain $%
\Omega _{\varepsilon }=\Omega \backslash \overline{T_{\varepsilon }}$ is
supposed to be the reference configuration of some linear elastic,
homogeneous and isotropic material, thus satisfying the following Hooke's
law 
\begin{equation}
\sigma _{ij}\left( u\right) =\lambda e_{mm}\left( u\right) \delta _{ij}+2\mu
e_{ij}\left( u\right) ,\text{\qquad }i\text{, }j\text{, }m=1,2,3,
\label{Hooke}
\end{equation}%
where the summation convention has been used with respect to repeated
indices, $\lambda $ and $\mu $ are the Lam\'{e} coefficients of the
material, satisfying : $\mu >0$ and $\lambda \geq 0$, $\delta _{ij}$ is
Kronecker's symbol and $e\left( u\right) $ is the linearized deformation
tensor, the components of which are given by : $e_{ij}\left( u\right) =\frac{%
1}{2}\left( \frac{\partial u_{j}}{\partial x_{i}}+\frac{\partial u_{i}}{%
\partial x_{j}}\right) $.\FRAME{fhFU}{8.2373cm}{4.916cm}{0pt}{\Qcb{The
domain $\Omega $ and the cylinders $T_{\protect\varepsilon }^{k}$.}}{}{%
jar8.gif}{\special{language "Scientific Word";type
"GRAPHIC";maintain-aspect-ratio TRUE;display "USEDEF";valid_file "F";width
8.2373cm;height 4.916cm;depth 0pt;original-width 4.0205in;original-height
2.3851in;cropleft "0";croptop "1";cropright "1";cropbottom "0";filename
'JAR8.gif';file-properties "XNPEU";}}

We suppose that $T_\varepsilon $ is the reference configuration of some
linear elastic, homogeneous and isotropic material satisfying Hooke's law 
\begin{equation}  \label{Hookeps}
\sigma _{ij}^\varepsilon \left( u\right) =\lambda ^\varepsilon e_{mm}\left(
u\right) \delta _{ij}+2\mu ^\varepsilon e_{ij}\left( u\right) ,\text{\qquad }%
i\text{, }j\text{, }m=1,2,3,
\end{equation}
where the Lam\'e coefficients $\lambda ^\varepsilon \geq 0$ and $\mu
^\varepsilon >0$ depend on $\varepsilon $ and satisfy 
\begin{equation}  \label{hypomu}
\exists c>0,\text{ }\forall \varepsilon >0:\mu ^\varepsilon \geq c.
\end{equation}

The structure $\Omega $ built with these two elastic materials is submitted
to some volumic forces the density of which $f=\left( f_1,f_2,f_3\right) $
belongs to $L^2\left( \Omega ,\mathbf{R}^3\right) $. We suppose that the
structure is held fixed along $\Gamma _1$ and that the tractions are equal
to 0 on the rest of the boundary : $\sigma _{ij}\left( u^\varepsilon \right)
n_j=0$, $i,j=1,2,3$, where $n$ is the unit outer normal to the boundary. Let
us introduce the functional $F^\varepsilon $ defined on $H^1\left( \Omega ,%
\mathbf{R}^3\right) $ by:

\begin{equation}  \label{Feps}
F^\varepsilon \left( u\right) =\left\{ 
\begin{array}{ll}
\dint_{\Omega _\varepsilon }\sigma _{ij}\left( u\right) e_{ij}\left(
u\right) dx+\dint_{T_\varepsilon }\sigma _{ij}^\varepsilon \left( u\right)
e_{ij}\left( u\right) dx & \text{if }u\in H_{\Gamma _1}^1\left( \Omega ,%
\mathbf{R}^3\right) \\ 
+\infty & \text{otherwise,}%
\end{array}
\right.
\end{equation}
with : $H_{\Gamma _1}^1\left( \Omega ,\mathbf{R}^3\right) =\left\{ u\in
H^1\left( \Omega ,\mathbf{R}^3\right) \mid u=0\text{ on }\Gamma _1\right\} $%
. The problem under consideration can be associated to the minimization
problem involving the functional $F^\varepsilon $, as indicated in the
following

\begin{lemma}
\label{1.1}

\begin{enumerate}
\item The minimization problem: 
\begin{equation}
\underset{u\in H^{1}\left( \Omega ,\mathbf{R}^{3}\right) }{\min }\left\{
F^{\varepsilon }\left( u\right) -2\dint_{\Omega }f.udx\right\} ,
\label{Mineps}
\end{equation}%
admits a unique solution $u^{\varepsilon }$ belonging to $H_{\Gamma
_{1}}^{1}\left( \Omega ,\mathbf{R}^{3}\right) $ and which satisfies the
variational formulation: 
\begin{equation}
\dint_{\Omega _{\varepsilon }}\sigma _{ij}\left( u^{\varepsilon }\right)
e_{ij}\left( u\right) dx+\dint_{T_{\varepsilon }}\sigma _{ij}^{\varepsilon
}\left( u^{\varepsilon }\right) e_{ij}\left( u\right) dx=\dint_{\Omega
}f.udx,\quad \forall u\in H_{\Gamma _{1}}^{1}\left( \Omega ,\mathbf{R}%
^{3}\right)   \label{formvar}
\end{equation}%
and is a weak solution of the problem: 
\begin{equation}
\left\{ 
\begin{array}{rcll}
-\sigma _{ij,j}\left( u^{\varepsilon }\right)  & = & f_{i} & \text{in }%
\Omega _{\varepsilon } \\ 
-\sigma _{ij,j}^{\varepsilon }\left( u^{\varepsilon }\right)  & = & f_{i} & 
\text{in }T_{\varepsilon } \\ 
u^{\varepsilon } & = & 0 & \text{on }\Gamma _{1} \\ 
\sigma _{ij}\left( u^{\varepsilon }\right) n_{j} & = & 0 & \text{on }%
\partial \Omega \setminus \Gamma _{1}.%
\end{array}%
\right.   \label{Peps}
\end{equation}

\item The sequence $\left( u^{\varepsilon }\right) _{\varepsilon }$ is
bounded in $H^{1}\left( \Omega ,\mathbf{R}^{3}\right) $.

\item Assume that : $\sup_{\varepsilon }\left( -\varepsilon ^{2}\ln \left(
r_{\varepsilon }\right) \right) <+\infty $.\ Then, $\sup_{\varepsilon
}\left( \left( \int_{T_{\varepsilon }}\left\vert u^{\varepsilon }\right\vert
^{2}dx\right) /\left\vert T_{\varepsilon }\right\vert \right) $ is finite
and if $R^{\varepsilon }\left( u^{\varepsilon }\right) $ is the rescaled
restriction of $u^{\varepsilon }$ to the fibers defined by: 
\begin{equation}
R^{\varepsilon }\left( u^{\varepsilon }\right) =\dfrac{\left\vert \Omega
\right\vert }{\left\vert T_{\varepsilon }\right\vert }u^{\varepsilon }%
\mathbf{1}_{T_{\varepsilon }},  \label{Reps}
\end{equation}%
where $\left\vert \Omega \right\vert $ means the volume of $\Omega $ and $%
\mathbf{1}_{T_{\varepsilon }}$ denotes the characteristic function of $%
T_{\varepsilon }$, the sequence $\left( R^{\varepsilon }\left(
u^{\varepsilon }\right) \right) _{\varepsilon }$ is bounded in $L^{1}\left( 
\mathbf{R}^{3},\mathbf{R}^{3}\right) $.
\end{enumerate}
\end{lemma}

\begin{proof}
\textit{1.} Because $\lambda ^{\varepsilon }$ is nonnegative, we write for
every $u$ in $H_{\Gamma _{1}}^{1}\left( \Omega ,\mathbf{R}^{3}\right) $%
\[
F^{\varepsilon }\left( u\right) \geq \inf \left( 2\mu ,2\mu ^{\varepsilon
}\right) \dint_{\Omega }e_{ij}\left( u\right) e_{ij}\left( u\right) dx\geq
C\inf \left( 2\mu ,2\mu ^{\varepsilon }\right) \dint_{\Omega }\left\vert
\nabla u\right\vert ^{2}dx,
\]%
using the classical Korn's inequality, because $u$ vanishes on $\Gamma _{1}$%
. The hypothesis (\ref{hypomu}) and this inequality imply that $%
F^{\varepsilon }$ is coercive on $H^{1}\left( \Omega ,\mathbf{R}^{3}\right) $%
. Moreover, $F^{\varepsilon }$ is lower semi-continuous for the weak
topology of $H_{\Gamma _{1}}^{1}\left( \Omega ,\mathbf{R}^{3}\right) $ and
is not identically equal to $+\infty $.\ Thus, classical convex analysis
results imply the existence and the uniqueness of a minimizer $%
u^{\varepsilon }$ of $F^{\varepsilon }$ on $H_{\Gamma _{1}}^{1}\left( \Omega
,\mathbf{R}^{3}\right) $, which satisfies the variational formulation (\ref%
{formvar}) and, thus, is a weak solution of (\ref{Peps}).

\noindent \textit{2.} We observe that : $F^{\varepsilon }\left(
u^{\varepsilon }\right) -2\int_{\Omega }f.u^{\varepsilon }dx\leq
F^{\varepsilon }\left( 0\right) =0$, which implies, using the preceding
inequality, that
\end{proof}

\[
C\inf \left( 2\mu ,2\mu ^{\varepsilon }\right) \dint_{\Omega }\left\vert
\nabla u^{\varepsilon }\right\vert ^{2}dx\leq 2\left\Vert f\right\Vert
_{L^{2}\left( \Omega \right) }\left\Vert u^{\varepsilon }\right\Vert
_{L^{2}\left( \Omega \right) }.
\]

Using Poincar\'{e}'s inequality, we thus deduce that $\left( u^{\varepsilon
}\right) _{\varepsilon }$ is bounded in $H_{\Gamma _{1}}^{1}\left( \Omega ,%
\mathbf{R}^{3}\right) $.

\noindent \textit{3.} Before proving this assertion, let us first recall the
following estimate, which has been proved in \cite{Pid-Sep}

\begin{lemma}
\label{1.2}There exists some positive constant $C$ such that, for every $u$
in $H^{1}\left( \Omega ,\mathbf{R}^{3}\right) $, one has : 
\begin{equation}
\dfrac{1}{\left\vert T_{\varepsilon }\right\vert }\dint_{T_{\varepsilon
}}u^{2}dx\leq C\left( \dint_{\Omega }\left\vert \nabla u\right\vert
^{2}dx-\varepsilon ^{2}\ln \left( r_{\varepsilon }\right) +\varepsilon
^{2}\right) .  \label{Korn}
\end{equation}
\end{lemma}

\begin{proof}
We first define : $u^{\prime }\left( r,\theta ,z\right) :=u\left(
\varepsilon k_{1}+r\cos \left( \theta \right) ,\varepsilon k_{2}+r\cos
\left( \theta \right) ,z\right) $, in the fiber centred at $\left(
\varepsilon k_{1},\varepsilon k_{2}\right) $. Then, we observe that, for
every $r_{1}\leq r_{2}<\varepsilon /2$%
\[
\begin{array}{c}
u^{\prime }\left( r_{2},\theta ,z\right) -u^{\prime }\left( r_{1},\theta
,z\right) =\left( r_{2}-r_{1}\right) \dint_{0}^{1}\dfrac{\partial u^{\prime }%
}{\partial r}\left( \left( 1-t\right) r_{1}+tr_{2}\right) \dfrac{\sqrt{%
\left( 1-t\right) r_{1}+tr_{2}}}{\sqrt{\left( 1-t\right) r_{1}+tr_{2}}}dt \\ 
\Rightarrow \left( u^{\prime }\left( r_{2},\theta ,z\right) -u^{\prime
}\left( r_{1},\theta ,z\right) \right) ^{2}\leq \left( \ln \left(
r_{2}\right) -\ln \left( r_{1}\right) \right) \dint_{r_{1}}^{r_{2}}\left( 
\dfrac{\partial u^{\prime }}{\partial r}\right) ^{2}rdr.%
\end{array}%
\]

Defining : $f\left( r\right) =\sum_{k\in K\left( \varepsilon \right)
}\int_{0}^{L}\int_{0}^{2\pi }\left( u^{\prime }\right) ^{2}\left( r,\theta
,z\right) d\theta dz$, the previous inequality implies : $f\left(
r_{1}\right) \leq 2f\left( r_{2}\right) +2\left\Vert \nabla u\right\Vert
_{L^{2}\left( \Omega ,\mathbf{R}^{3}\right) }^{2}\ln \left(
r_{2}/r_{1}\right) $, which implies, for every $r_{2}$ in $\left[
\varepsilon /4,\varepsilon /2\right] $%
\[
\begin{array}{ccl}
\dfrac{1}{\left\vert T_{\varepsilon }\right\vert }\dint_{T_{\varepsilon
}}u^{2}dx & = & \dfrac{1}{\left\vert T_{\varepsilon }\right\vert }%
\dint_{0}^{r_{\varepsilon }}f\left( r\right) rdr \\ 
& \leq  & \dfrac{2}{\left\vert T_{\varepsilon }\right\vert }%
\dint_{0}^{r_{\varepsilon }}\left( f\left( r_{2}\right) +\left\Vert \nabla
u\right\Vert _{L^{2}\left( \Omega ,\mathbf{R}^{3}\right) }^{2}\left( \ln
\left( r_{2}\right) -\ln \left( r\right) \right) \right) rdr \\ 
& \leq  & \dfrac{C\varepsilon ^{2}}{\left( r_{\varepsilon }\right) ^{2}}%
\left( f\left( r_{2}\right) \left( r_{\varepsilon }\right) ^{2}+\left\Vert
\nabla u\right\Vert _{L^{2}\left( \Omega ,\mathbf{R}^{3}\right) }^{2}\left(
\left( r_{\varepsilon }\right) ^{2}-\dfrac{\left( r_{\varepsilon }\right)
^{2}}{2}\ln \left( r_{\varepsilon }\right) +\dfrac{\left( r_{\varepsilon
}\right) ^{2}}{4}\right) \right)  \\ 
& \leq  & C\left( f\left( r_{2}\right) \varepsilon ^{2}+\left\Vert \nabla
u\right\Vert _{L^{2}\left( \Omega ,\mathbf{R}^{3}\right) }^{2}\varepsilon
^{2}-\dfrac{\varepsilon ^{2}}{2}\ln \left( r_{\varepsilon }\right) +\dfrac{%
\varepsilon ^{2}}{4}\right)  \\ 
& \leq  & C\left( 4f\left( r_{2}\right) \varepsilon r_{2}+\left\Vert \nabla
u\right\Vert _{L^{2}\left( \Omega ,\mathbf{R}^{3}\right) }^{2}\varepsilon
^{2}-\dfrac{\varepsilon ^{2}}{2}\ln \left( r_{\varepsilon }\right) +\dfrac{%
\varepsilon ^{2}}{4}\right) 
\end{array}%
\]%
and then, taking the mean value of this inequality with respect to $r_{2}$
in $\left[ \varepsilon /4,\varepsilon /2\right] $
\end{proof}

\[
\begin{array}{ccl}
\dfrac{1}{\left\vert T_{\varepsilon }\right\vert }\dint_{T_{\varepsilon
}}u^{2}dx & \leq  & C\left( 16\dint_{\varepsilon /4}^{\varepsilon /2}f\left(
r\right) rdr+\left\Vert \nabla u\right\Vert _{L^{2}\left( \Omega ,\mathbf{R}%
^{3}\right) }^{2}\varepsilon ^{2}-\dfrac{\varepsilon ^{2}}{2}\ln \left(
r_{\varepsilon }\right) +\dfrac{\varepsilon ^{2}}{4}\right)  \\ 
& \leq  & C\left( \left( 16+\varepsilon ^{2}\right) \left\Vert \nabla
u\right\Vert _{L^{2}\left( \Omega ,\mathbf{R}^{3}\right) }^{2}-\dfrac{%
\varepsilon ^{2}}{2}\ln \left( r_{\varepsilon }\right) +\dfrac{\varepsilon
^{2}}{4}\right) .\qquad \square 
\end{array}%
\]

Coming back to the proof of Lemma \ref{1.1}, we observe that Lemma \ref{1.2}
implies that $\sup_{\varepsilon }\left( \left( \int_{T_{\varepsilon
}}\left\vert u^{\varepsilon }\right\vert ^{2}dx\right) /\left\vert
T_{\varepsilon }\right\vert \right) $ is finite, as soon as $%
\sup_{\varepsilon }\left( -\varepsilon ^{2}\ln \left( r_{\varepsilon
}\right) \right) <+\infty $.\ Then, using Cau- chy-Schwarz inequality, we
finally prove that the quantity $\left( \int_{\mathbf{R}^{3}}\left\vert
R^{\varepsilon }\left( u^{\varepsilon }\right) \right\vert dx\right)
_{\varepsilon }$ is bounded, which ends the proof of Lemma \ref{1.1}.\qquad $%
\square $

In the sequel, we will assume that the hypothesis $\sup_{\varepsilon }\left(
-\varepsilon ^{2}\ln \left( r_{\varepsilon }\right) \right) <+\infty $ is
always satisfied.

Our purpose is to describe the asymptotic behaviour of $\left(
u^{\varepsilon }\right) _{\varepsilon }$ and that of $\left( R^{\varepsilon
}\left( u^{\varepsilon }\right) \right) _{\varepsilon }$, when $\varepsilon $
goes to 0. This will be obtained using epi-convergence arguments, that is
studying the asymptotic behaviour of the sequence $\left( F^{\varepsilon
}\right) _{\varepsilon }$, when $\varepsilon $ goes to 0.\ We will first
suppose that the coefficients $\lambda _{o}$ and $\mu _{o}$, defined by%
\begin{equation}
\lambda _{o}=\ \underset{\varepsilon \rightarrow 0}{\lim }\dfrac{\lambda
^{\varepsilon }\left( r_{\varepsilon }\right) ^{2}}{\varepsilon ^{2}}\text{, 
}\mu _{o}=\text{ }\underset{\varepsilon \rightarrow 0}{\lim }\dfrac{\mu
^{\varepsilon }\left( r_{\varepsilon }\right) ^{2}}{\varepsilon ^{2}}.
\label{lambdas}
\end{equation}%
are finite and $\mu _{o}$ is positive.\ Thanks to the properties of the
epi-convergence, we then derive the asymptotic behaviour of the solution in
many other cases.

This kind of reinforcement problems follows earlier works like \cite{Bel-Bou}%
, \cite{Cai-Din}, \cite{Pid-Sep}, for example.\ However, the works \cite%
{Bel-Bou} and \cite{Cai-Din} were dealing with scalar problems (also
involving the $p$-laplacian operator).\ The work \cite{Pid-Sep} is dealing
with linear elasticity problems but assuming another scaling of the
coefficients, which will be described later on in the present work. The work 
\cite{Khr} deals with the homogenization of composite media evoking the
vectorial case. See also \cite{Mos} for similar phenomena in a quite general
situation.

\section{Construction and study of the test-functions}

We define%
\[
\begin{array}{rcl}
D & = & \left\{ \left( y_{1},y_{2}\right) \in \mathbf{R}^{2}\mid \left(
y_{1}\right) ^{2}+\left( y_{2}\right) ^{2}<1\right\}  \\ 
D\left( r,r^{\prime }\right)  & = & \left\{ \left( y_{1},y_{2}\right) \in 
\mathbf{R}^{2}\mid r^{2}<\left( y_{1}\right) ^{2}+\left( y_{2}\right)
^{2}<r^{\prime 2}\right\}  \\ 
S_{r} & = & \left\{ \left( y_{1},y_{2}\right) \in \mathbf{R}^{2}\mid \left(
y_{1}\right) ^{2}+\left( y_{2}\right) ^{2}=r^{2}\right\} 
\end{array}%
\]%
for $0<r<r^{\prime }$, and for every $k=\left( k_{1},k_{2}\right) $ in $%
\mathbf{Z}^{2}$%
\[
\begin{array}{rcl}
B_{\varepsilon }^{k} & = & \left\{ \left( x_{1},x_{2},x_{3}\right) \mid
\left( x_{1}-k_{1}\varepsilon \right) ^{2}+\left( x_{2}-k_{2}\varepsilon
\right) ^{2}<\left( s_{\varepsilon }\right) ^{2}\text{, }x_{3}\in \left] 0,L%
\right[ \right\}  \\ 
C_{\varepsilon }^{k} & = & \left\{ \left( x_{1},x_{2},x_{3}\right) \mid
\left( r_{\varepsilon }\right) ^{2}<\left( x_{1}-k_{1}\varepsilon \right)
^{2}+\left( x_{2}-k_{2}\varepsilon \right) ^{2}<\left( s_{\varepsilon
}\right) ^{2}\text{, }x_{3}\in \left] 0,L\right[ \right\} ,%
\end{array}%
\]%
choosing $s_{\varepsilon }$ such that

\[
\underset{\varepsilon \rightarrow 0}{\lim }\dfrac{s_\varepsilon }\varepsilon
=0=\ \underset{\varepsilon \rightarrow 0}{\lim }\dfrac{r_\varepsilon }{%
s_\varepsilon }=0=\ \underset{\varepsilon \rightarrow 0}{\lim }\varepsilon
^2\ln ^2s_\varepsilon . 
\]

Finally, we denote: $B_{\varepsilon }=\cup _{k}B_{\varepsilon }^{k}$, $%
C_{\varepsilon }=\cup _{k}C_{\varepsilon }^{k}$.

We introduce the solution $w^{m}=\left( w_{1}^{m},w_{2}^{m}\right) $, $m=1,2$%
, of the linear plane elasticity problems%
\begin{equation}
\left\{ 
\begin{array}{rcll}
\sigma _{ij,j}\left( w^{m}\right) \left( y\right)  & = & 0 & \forall y\in 
\mathbf{R}^{2}\backslash \overline{D}\text{, }i\text{, }j=1,2 \\ 
w^{m}\left( y\right)  & = & 0 & \text{on }S_{1} \\ 
w_{m}^{m}\left( y\right)  & \simeq  & -\ln \left\vert y\right\vert +Cte & 
\text{when }\left\vert y\right\vert \rightarrow \infty  \\ 
\left\vert w_{p}^{m}\right\vert \left( y\right)  & \leq  & Cte & \text{when }%
\left\{ 
\begin{array}{l}
p=2\text{ if }m=1 \\ 
p=1\text{ if }m=2,%
\end{array}%
\right. 
\end{array}%
\right.   \label{Pm}
\end{equation}%
where: $\sigma _{ij}\left( w^{m}\right) =\lambda e_{ij}\left( w^{m}\right)
+2\mu e_{ij}\left( w^{m}\right) $.\ Thanks to the potential theory methods,
described for example in \cite{Sok}, the solution $w^{m}$ of (\ref{Pm}) can
be computed as%
\[
\left\{ 
\begin{array}{ccl}
w_{1}^{1}\left( y_{1},y_{2}\right)  & = & -\ln \left\vert y\right\vert +%
\dfrac{\left( y_{2}\right) ^{2}-\left( y_{1}\right) ^{2}}{2\kappa \left\vert
y\right\vert ^{2}}-\dfrac{\left( y_{2}\right) ^{2}-\left( y_{1}\right) ^{2}}{%
2\kappa \left\vert y\right\vert ^{4}} \\ 
w_{2}^{1}\left( y_{1},y_{2}\right)  & = & \dfrac{y_{2}y_{1}}{\kappa
\left\vert y\right\vert ^{2}}-\dfrac{y_{2}y_{1}}{\kappa \left\vert
y\right\vert ^{4}} \\ 
w_{1}^{2}\left( y_{1},y_{2}\right)  & = & \dfrac{y_{2}y_{1}}{\kappa
\left\vert y\right\vert ^{2}}-\dfrac{y_{2}y_{1}}{\left\vert y\right\vert ^{4}%
} \\ 
w_{2}^{2}\left( y_{1},y_{2}\right)  & = & -\ln \left\vert y\right\vert -%
\dfrac{\left( y_{2}\right) ^{2}-\left( y_{1}\right) ^{2}}{2\kappa \left\vert
y\right\vert ^{2}}+\dfrac{\left( y_{2}\right) ^{2}-\left( y_{1}\right) ^{2}}{%
2\kappa \left\vert y\right\vert ^{4}},%
\end{array}%
\right. 
\]%
with : $\kappa =\left( \lambda +3\mu \right) /\left( \lambda +\mu \right) $.
We also introduce the function $w(y_{1},y_{2})=-\ln \left\vert y\right\vert $%
, which is harmonic in $\mathbf{R}^{2}\setminus \left\{ 0\right\} $ and
verifies the following properties 
\[
w_{\mid S_{1}}=0\text{, }\underset{\left\vert y\right\vert \rightarrow
\infty }{\lim }\dfrac{w\left( y_{1},y_{2}\right) }{\ln \left\vert
y\right\vert }=-1\text{, }\dint_{S_{1}}\dfrac{\partial w}{\partial n}d\sigma
=2\pi .
\]

Let us observe that

\begin{lemma}
\label{2.1}One has the following convergences:

\begin{enumerate}
\item $\underset{R\rightarrow +\infty }{\lim }\dfrac{1}{\ln R}\dint_{D\left(
1,R\right) }\sigma _{ij}\left( w^{m}\right) e_{ij}\left( w^{l}\right) dy=%
\dfrac{2\pi \mu \left( 1+\kappa \right) }{\kappa }\delta _{lm}.$

\item $\underset{R\rightarrow +\infty }{\lim }\dfrac{1}{\ln R}\dint_{D\left(
1,R\right) }\left\vert \nabla w\right\vert ^{2}dy=2\pi $,
\end{enumerate}
\end{lemma}

\begin{proof}
The proof is trivial.\qquad $\square $
\end{proof}

Using the solutions of these plane problems, we now build the functions $%
w_{\varepsilon }^{mk}$, for every $k=\left( k_{1},k_{2}\right) $ as

\[
\begin{array}{rcl}
w_{\varepsilon }^{\alpha k}\left( x_{1},x_{2}\right)  & = & \dfrac{-1}{\ln
r_{\varepsilon }}\left( 
\begin{array}{c}
w_{1}^{\alpha }\left( \dfrac{x_{1}-k_{1}\varepsilon }{r_{\varepsilon }},%
\dfrac{x_{2}-k_{2}\varepsilon }{r_{\varepsilon }}\right)  \\ 
w_{2}^{\alpha }\left( \dfrac{x_{1}-k_{1}\varepsilon }{r_{\varepsilon }},%
\dfrac{x_{2}-k_{2}\varepsilon }{r_{\varepsilon }}\right)  \\ 
0%
\end{array}%
\right)  \\ 
w_{\varepsilon }^{3k}\left( x_{1},x_{2}\right)  & = & \dfrac{-1}{\ln
r_{\varepsilon }}\left( 
\begin{array}{c}
0 \\ 
0 \\ 
w\left( \dfrac{x_{1}-k_{1}\varepsilon }{r_{\varepsilon }},\dfrac{%
x_{2}-k_{2}\varepsilon }{r_{\varepsilon }}\right) 
\end{array}%
\right) ,%
\end{array}%
\]%
$\alpha =1$, $2$. These functions $w_{\varepsilon }^{mk}$ satisfy the
following properties.

\begin{lemma}
\label{2.2}There exist two positive constants $C_{0}$ and $C_{1}$,
independant of $\varepsilon $, such that:

\begin{enumerate}
\item $\left\vert e_{m}-w_{\varepsilon }^{mk}\right\vert ^{2}\leq C_{0}%
\dfrac{\ln ^{2}\left( R_{\varepsilon }^{k}\right) +1}{\ln ^{2}\left(
r_{\varepsilon }\right) }$, in $B_{\varepsilon }^{k},$

\item $\left\vert \dfrac{\partial w_{\varepsilon }^{mk}}{\partial x_{i}}%
\right\vert ^{2}\leq \dfrac{C_{1}}{\left( R_{\varepsilon }^{k}\right)
^{2}\ln ^{2}\left( r_{\varepsilon }\right) }$, in $B_{\varepsilon }^{k}$, $%
i=1,2,3$,
\end{enumerate}

\noindent where $e_{m}$ is the $m$-th vector of the canonical basis of $%
\mathbf{R}^{3}$ and 
\[
\left( R_{\varepsilon }^{k}\right) ^{2}=\left( x_{1}-k_{1}\varepsilon
\right) ^{2}+\left( x_{2}-k_{2}\varepsilon \right) ^{2}.
\]
\end{lemma}

\begin{proof}
Immediate, thanks to the expression of $w_{\varepsilon }^{mk}$.\qquad $%
\square $
\end{proof}

\begin{lemma}
\label{2.3}If $\gamma :=\lim_{\varepsilon \rightarrow 0}\left( -1/\left(
\varepsilon ^{2}\ln r_{\varepsilon }\right) \right) $ is finite, then:

\begin{enumerate}
\item For every $m$ and $l$, one has 
\[
\underset{\varepsilon \rightarrow 0}{\lim }\dint\nolimits_{B_{\varepsilon
}}\sigma _{ij}\left( w_{\varepsilon }^{mk}\right) e_{ij}\left(
w_{\varepsilon }^{lk}\right) dx=\left\{ 
\begin{array}{ll}
\dfrac{2\pi \gamma \mu \left( 1+\kappa \right) }{\kappa }\left\vert \Omega
\right\vert \delta _{lm} & m,l=1,2 \\ 
0 & l=3\text{, }m=1,2 \\ 
2\pi \gamma \mu \left\vert \Omega \right\vert  & m,l=3.%
\end{array}%
\right. 
\]

\item Let $\varphi $ be any element of $C^{1}\left( \overline{\Omega }%
\right) $. Then 
\[
\underset{\varepsilon \rightarrow 0}{\lim }\dint\nolimits_{B_{\varepsilon
}}\sigma _{ij}\left( w_{\varepsilon }^{mk}\right) e_{ij}\left(
w_{\varepsilon }^{lk}\right) \varphi dx=\left\{ 
\begin{array}{ll}
\dfrac{2\pi \gamma \mu \left( 1+\kappa \right) }{\kappa }\delta
_{lm}\dint_{\Omega }\varphi dx & m,l=1,2 \\ 
0 & l=3\text{, }m=1,2 \\ 
2\pi \gamma \mu \dint_{\Omega }\varphi dx & m,l=3.%
\end{array}%
\right. 
\]

\item Let $\varphi _{\varepsilon }^{k}$ be the truncation function defined
by 
\begin{equation}
\varphi _{\varepsilon }^{k}\left( x\right) =\varphi _{\varepsilon
}^{k}\left( x_{1},x_{2}\right) =\left\{ 
\begin{array}{ll}
\dfrac{-4}{3\left( s_{\varepsilon }\right) ^{2}}\left( \left( R_{\varepsilon
}^{k}\right) ^{2}-\left( s_{\varepsilon }\right) ^{2}\right)  & \text{if }%
\dfrac{s_{\varepsilon }}{2}\leq R_{\varepsilon }^{k}\leq s_{\varepsilon } \\ 
1 & \text{if }R_{\varepsilon }^{k}\leq \dfrac{s_{\varepsilon }}{2} \\ 
0 & \text{if }R_{\varepsilon }^{k}\geq s_{\varepsilon }%
\end{array}%
\right.   \label{phiepsk}
\end{equation}%
and $z_{\varepsilon }^{m}$ the function defined by 
\begin{equation}
z_{\varepsilon }^{m}\left( x\right) =\left\{ 
\begin{array}{ll}
\varphi _{\varepsilon }^{k}\left( x\right) \left( e_{m}-w_{\varepsilon
}^{mk}\right) \left( x\right)  & \forall x\in B_{\varepsilon }^{k}\text{, }%
\forall k \\ 
0 & \forall x\in \Omega \backslash \overline{B_{\varepsilon }}.%
\end{array}%
\right.   \label{zepsm}
\end{equation}%
Then $\left( z_{\varepsilon }^{m}\right) _{\mid T_{\varepsilon }}=e_{m}$, $%
\left( z_{\varepsilon }^{m}\right) _{\varepsilon }$ converges to 0 in the
weak topology of $H^{1}\left( \Omega ,\mathbf{R}^{3}\right) $ and 
\[
\underset{\varepsilon \rightarrow 0}{\lim }\dint\nolimits_{\Omega }\sigma
_{ij}\left( z_{\varepsilon }^{m}\right) e_{ij}\left( z_{\varepsilon
}^{l}\right) dx=\left\{ 
\begin{array}{ll}
\dfrac{2\pi \gamma \mu \left( 1+\kappa \right) }{\kappa }\left\vert \Omega
\right\vert \delta _{lm} & \text{if }m,l=1,2 \\ 
0 & \text{if }l=3\text{, }m=1,2 \\ 
2\pi \gamma \mu \left\vert \Omega \right\vert  & \text{if }m,l=3.%
\end{array}%
\right. 
\]
\end{enumerate}
\end{lemma}

\begin{proof}
\textit{1}. Using Hooke's law, the above expression of $w_{\varepsilon }^{mk}
$ and the estimates given in Lemma \ref{2.2}, one has, for $m,l=1,3$ 
\[
\underset{\varepsilon \rightarrow 0}{\lim }\dint\nolimits_{C_{\varepsilon
}}\sigma _{ij}\left( w_{\varepsilon }^{mk}\right) e_{ij}\left(
w_{\varepsilon }^{lk}\right) dx=\dfrac{\left\vert \Omega \right\vert }{%
\varepsilon ^{2}\ln ^{2}r_{\varepsilon }}\dint\nolimits_{D\left(
1,s_{\varepsilon }/r_{\varepsilon }\right) }\sigma _{ij}\left( w^{m}\right)
e_{ij}\left( w^{l}\right) dy_{1}dy_{2}+o_{\varepsilon },
\]%
where: $y_{1}=\left( x_{1}-k_{1}\varepsilon \right) /r_{\varepsilon }$, $%
y_{2}=\left( x_{2}-k\varepsilon \right) /r_{\varepsilon }$, $\sigma _{ij}$
and $e_{ij}$ respectively denote the stress and the deformation tensors in
the plane, with the Lam\'{e} coefficients $\lambda $ and $\mu $ and $%
\lim_{\varepsilon \rightarrow 0}o_{\varepsilon }=0$. One deduces from Lemma %
\ref{2.1}, through the definition of $s_{\varepsilon }$ that 
\[
\underset{\varepsilon \rightarrow 0}{\lim }\dfrac{-1}{\ln r_{\varepsilon }}%
\dint\nolimits_{D\left( 1,s_{\varepsilon }/r_{\varepsilon }\right) }\sigma
_{ij}\left( w^{m}\right) e_{ij}\left( w^{l}\right) dy_{1}dy_{2}=\dfrac{2\pi
\mu \left( 1+\kappa \right) }{\kappa }\delta _{ml},
\]%
the other cases being treated in a similar way. We conclude, using the
definition of $\gamma $.

\noindent \textit{2}. The smoothness of $\varphi $ implies that for every $%
\left( x_{1},x_{2},x_{3}\right) $ in $C_{\varepsilon }^{k}$ we have : $%
\varphi \left( x_{1},x_{2},x_{3}\right) =\varphi \left( k_{1}\varepsilon
,k_{2}\varepsilon ,x_{3}\right) +O\left( R_{\varepsilon }^{k}\right) $,
which implies 
\[
\begin{array}{l}
\dint\nolimits_{C_{\varepsilon }}\sigma _{ij}\left( w_{\varepsilon
}^{mk}\right) e_{ij}\left( w_{\varepsilon }^{lk}\right) \varphi dx \\ 
=\dfrac{1}{\varepsilon ^{2}\ln ^{2}r_{\varepsilon }}\left(
\dint\nolimits_{D\left( 1,s_{\varepsilon }/r_{\varepsilon }\right) }\sigma
_{ij}\left( w^{m}\right) e_{ij}\left( w^{l}\right) dy_{1}dy_{2}\left(
\dsum\limits_{k}\varepsilon ^{2}\dint_{0}^{L}\varphi \left( k_{1}\varepsilon
,k_{2}\varepsilon ,x_{3}\right) dx_{3}\right) \right) +o_{\varepsilon }.%
\end{array}%
\]

But the smoothness of $\varphi $ also implies 
\[
\underset{\varepsilon \rightarrow 0}{\lim }\dsum\limits_{k}\varepsilon
^{2}\dint_{0}^{L}\varphi \left( k_{1}\varepsilon ,k_{2}\varepsilon
,x_{3}\right) dx_{3}=\dint_{\Omega }\varphi dx,
\]%
from which we conclude, using the first assertion.

\noindent \textit{3}. We observe that $\varphi _{\varepsilon }^{k}\equiv 0$
in $\Omega \setminus \overline{B_{\varepsilon }}$ and $w_{\varepsilon
}^{mk}\equiv 0$ in $T_{\varepsilon }$. Then we compute 
\[
\begin{array}{ccl}
\dint\nolimits_{\Omega }\sigma _{ij}\left( z_{\varepsilon }^{m}\right)
e_{ij}\left( z_{\varepsilon }^{l}\right) dx & = & \dsum\limits_{k}\dint%
\nolimits_{C_{\varepsilon }^{k}}\sigma _{ij}\left( w_{\varepsilon
}^{mk}\right) e_{ij}\left( w_{\varepsilon }^{lk}\right) \left( \varphi
_{\varepsilon }^{k}\right) ^{2}dx \\ 
&  & -2\dsum\limits_{k}\dint\nolimits_{C_{\varepsilon }^{k}\cap \left\{
s_{\varepsilon }/2<R_{\varepsilon }^{k}<s_{\varepsilon }\right\} }\sigma
_{ij}\left( w_{\varepsilon }^{mk}\right) \dfrac{\partial \varphi
_{\varepsilon }^{k}}{\partial x_{i}}\left( e_{l}-w_{\varepsilon
}^{lk}\right) _{j}dx \\ 
&  & +\dsum\limits_{k}\dint\nolimits_{C_{\varepsilon }^{k}\cap \left\{
s_{\varepsilon }/2<R_{\varepsilon }^{k}<s_{\varepsilon }\right\} }\left(
e_{m}-w_{\varepsilon }^{mk}\right) _{i}\dfrac{\partial \varphi _{\varepsilon
}^{k}}{\partial x_{i}}\left( e_{l}-w_{\varepsilon }^{lk}\right) _{j}\dfrac{%
\partial \varphi _{\varepsilon }^{k}}{\partial x_{j}}dx.%
\end{array}%
\]

Thanks to Lemma \ref{2.2} and to the definition of $\varphi _{\varepsilon
}^{k}$, one can prove that the two last sums are respectively bounded by : $%
C\left\vert \ln s_{\varepsilon }\right\vert /\left( \varepsilon ^{2}\ln
^{2}r_{\varepsilon }\right) $ and $C\ln ^{2}s_{\varepsilon }/\left(
\varepsilon ^{2}\ln ^{2}r_{\varepsilon }\right) $. These two upper bounds
converge to $0$, because $\gamma $ is finite and thanks to the choice of $%
s_{\varepsilon }$. Moreover, the first term of the preceding equality can be
computed as 
\[
\begin{array}{r}
\dint\nolimits_{C_{\varepsilon }^{k}}\sigma _{ij}\left( w_{\varepsilon
}^{mk}\right) e_{ij}\left( w_{\varepsilon }^{lk}\right) \left( \varphi
_{\varepsilon }^{k}\right) ^{2}dx=\dint\nolimits_{C_{\varepsilon
}^{k}}\sigma _{ij}\left( w_{\varepsilon }^{mk}\right) e_{ij}\left(
w_{\varepsilon }^{lk}\right) dx\qquad  \\ 
+\dint\nolimits_{C_{\varepsilon }^{k}\cap \left\{ s_{\varepsilon
}/2<R_{\varepsilon }^{k}<s_{\varepsilon }\right\} }\sigma _{ij}\left(
w_{\varepsilon }^{mk}\right) e_{ij}\left( w_{\varepsilon }^{lk}\right)
\left( \left( \varphi _{\varepsilon }^{k}\right) ^{2}-1\right) dx%
\end{array}%
\]%
and using the definition (\ref{phiepsk}) of $\varphi _{\varepsilon }^{k}$ we
get 
\[
\begin{array}{r}
\left\vert \dint\nolimits_{C_{\varepsilon }^{k}\cap \left\{ s_{\varepsilon
}/2<R_{\varepsilon }^{k}<s_{\varepsilon }\right\} }\sigma _{ij}\left(
w_{\varepsilon }^{mk}\right) e_{ij}\left( w_{\varepsilon }^{lk}\right)
\left( \left( \varphi _{\varepsilon }^{k}\right) ^{2}-1\right) dx\right\vert
\qquad  \\ 
\leq \dint\nolimits_{C_{\varepsilon }^{k}\cap \left\{ s_{\varepsilon
}/2<R_{\varepsilon }^{k}<s_{\varepsilon }\right\} }\sigma _{ij}\left(
w_{\varepsilon }^{mk}\right) e_{ij}\left( w_{\varepsilon }^{lk}\right) dx.%
\end{array}%
\]

Thanks to the estimates of Lemma \ref{2.2}, we deduce 
\[
\underset{\varepsilon \rightarrow 0}{\lim }\dsum\limits_{k}\dint%
\nolimits_{C_{\varepsilon }^{k}\cap \left\{ s_{\varepsilon
}/2<R_{\varepsilon }^{k}<s_{\varepsilon }\right\} }\sigma _{ij}\left(
w_{\varepsilon }^{mk}\right) e_{ij}\left( w_{\varepsilon }^{lk}\right) dx=0,
\]%
which implies 
\[
\underset{\varepsilon \rightarrow 0}{\lim }\dint\nolimits_{\Omega }\sigma
_{ij}\left( z_{\varepsilon }^{m}\right) e_{ij}\left( z_{\varepsilon
}^{l}\right) dx=\ \underset{\varepsilon \rightarrow 0}{\lim }%
\dint\nolimits_{C_{\varepsilon }}\sigma _{ij}\left( w_{\varepsilon
}^{mk}\right) e_{ij}\left( w_{\varepsilon }^{lk}\right) dx.
\]

One concludes using the first assertion. Because $\left( z_{\varepsilon
}^{m}\right) _{\mid \Gamma _{1}}=0$, there exists some positive constant $C$
such that 
\[
\dint\nolimits_{\Omega }\left\vert \nabla z_{\varepsilon }^{m}\right\vert
^{2}dx\leq C\dint\nolimits_{\Omega }\sigma _{ij}\left( z_{\varepsilon
}^{m}\right) e_{ij}\left( z_{\varepsilon }^{m}\right) dx.
\]

Hence $\left( z_{\varepsilon }^{m}\right) _{\varepsilon }$ is bounded in $%
H^{1}\left( \Omega ,\mathbf{R}^{3}\right) $, which implies that a
subsequence still denoted $\left( z_{\varepsilon }^{m}\right) _{\varepsilon }
$ converges to some $z^{\ast }$ in the weak topology of\ $H^{1}\left( \Omega
,\mathbf{R}^{3}\right) $ and in the strong topology of\ $L^{2}\left( \Omega ,%
\mathbf{R}^{3}\right) $. We observe that $z_{\varepsilon }^{m}=0$ in $\Omega
\backslash \overline{B_{\varepsilon }}$ and because the sequence of
characteristic functions of $\Omega \backslash \overline{B_{\varepsilon }}$
converges to $1$ in the strong topology of $L^{2}\left( \Omega \right) $, we
infer that $z^{\ast }=0$. Hence $\left( z_{\varepsilon }^{m}\right)
_{\varepsilon }$ converges to 0 in the weak topology of\ $H^{1}\left( \Omega
,\mathbf{R}^{3}\right) $.\qquad $\square $
\end{proof}

\section{Convergence}

We define the topology $\tau $ which will be used throughout this paragraph
as

\[
u_\varepsilon \overset{\tau }{\underset{\varepsilon \rightarrow 0}{%
\rightharpoonup }}\left( u,v\right) \Leftrightarrow \left\{ 
\begin{array}{l}
u_\varepsilon \overset{w\text{-}H^1\left( \Omega ,\mathbf{R}^3\right) }{%
\underset{\varepsilon \rightarrow 0}{\rightharpoonup }}u \\ 
\text{and : }\forall \varphi \in C_0^0\left( \mathbf{R}^3\right)
:\dint_\Omega R^\varepsilon \left( u_\varepsilon \right) \varphi dx\underset{%
\varepsilon \rightarrow 0}{\rightarrow }\dint_\Omega v\varphi dx,%
\end{array}
\right. 
\]
where $w$-$H^1\left( \Omega ,\mathbf{R}^3\right) $ stands for the weak
topology of $H^1\left( \Omega ,\mathbf{R}^3\right) $ and $R^\varepsilon $ is
the rescaled restriction operator defined in (\ref{Reps}).

Our main result reads as follows

\begin{theorem}
\label{3.1}Suppose that $\gamma =\lim_{\varepsilon \rightarrow 0}\left(
-1/\left( \varepsilon ^{2}\ln r_{\varepsilon }\right) \right) $ is finite, $%
\lambda _{o}$ and $\mu _{o}$ are finite and $\mu _{o}$ is positive. Then,
the sequence $\left( F^{\varepsilon }\right) _{\varepsilon }$ epi-converges
in the topology $\tau $ to the functional $F^{o}$ defined on $H^{1}\left(
\Omega ,\mathbf{R}^{3}\right) \times L^{1}\left( \Omega ,\mathbf{R}%
^{3}\right) $ by:%
\begin{equation}
F^{o}\left( u,v\right) =\left\{ 
\begin{array}{l}
\dint_{\Omega }\sigma _{ij}\left( u\right) e_{ij}\left( u\right) dx+2\pi
\gamma \dint_{\Omega }\left( v-u\right) ^{t}A\left( v-u\right) dx+\pi
E_{o}\dint_{\Omega }\left( e_{33}\left( v\right) \right) ^{2}dx, \\ 
\hfill \text{if }\left( u,v\right) \in H_{\Gamma _{1}}^{1}\left( \Omega ,%
\mathbf{R}^{3}\right) \times V \\ 
+\infty \hfill \text{otherwise,}%
\end{array}%
\right.   \label{Fo}
\end{equation}%
using the summation convention with respect to repeated indices and where $A$
is the diagonal matrix with : $A_{11}=\mu \left( 1+\kappa \right) /\kappa
=A_{22}$ and $A_{33}=\mu $, where $\kappa =\left( \lambda +3\mu \right)
/\left( \lambda +\mu \right) $, $E_{o}=\mu _{o}\left( 3\lambda _{o}+2\mu
_{o}\right) /\left( \lambda _{o}+\mu _{o}\right) $ and $V$ denotes the
subspace%
\[
V=\left\{ v\in L^{2}\left( \Omega ,\mathbf{R}^{3}\right) \mid v_{3\mid
\Gamma _{1}}=0\text{, }e_{33}\left( v\right) \in L^{2}\left( \Omega \right)
\right\} .
\]
\end{theorem}

As a consequence of this theorem and of the properties of the
epi-convergence (see \cite{Att} for a definition and the main properties of
this notion of convergence well-fitted to the description of the asymptotic
behaviour of the solution of minimization problems), one gets the following
asymptotic behaviour, when $\varepsilon $ goes to 0, of the solution $%
u^{\varepsilon }$ of (\ref{Mineps})

\begin{corollary}
\label{3.2}Under the hypotheses of Theorem \ref{3.1}, the solution $%
u^{\varepsilon }$ of (\ref{Mineps}) converges, in the topology $\tau $, to
the solution $\left( u^{o},v^{o}\right) $ in the space $H_{\Gamma
_{1}}^{1}\left( \Omega ,\mathbf{R}^{3}\right) \times V$ of the following
problem%
\begin{equation}
\left\{ 
\begin{array}{rcll}
-\sigma _{ij,j}\left( u^{o}\right) -2\gamma \pi A_{ij}\left(
v^{o}-u^{o}\right) _{j} & = & f_{i} & \text{in }\Omega \text{, }i=1,2,3 \\ 
u^{o} & = & 0 & \text{on }\Gamma _{1} \\ 
\sigma _{ij}\left( u^{o}\right) n_{j} & = & 0 & \text{on }\partial \omega
\times \left] 0,L\right[ \cup \Gamma _{2} \\ 
&  &  & \qquad i,j=1,2,3 \\ 
E_{o}\dfrac{\partial }{\partial x_{3}}\left( e_{33}\left( v^{o}\right)
\right)  & = & 2\gamma \mu \left( v^{o}-u^{o}\right) _{3} & \text{in }\Omega 
\\ 
v^{o} & = & 0 & \text{on }\Gamma _{1} \\ 
\left( u^{o}\right) _{\alpha } & = & \left( v^{o}\right) _{\alpha } & \text{%
in }\Omega \text{, }\alpha =1,2 \\ 
e_{33}\left( v^{o}\right)  & = & 0 & \text{on }\Gamma _{2}.%
\end{array}%
\right.   \label{equav}
\end{equation}

$\left( u^{o},v^{o}\right) $ is the unique solution of the minimization
problem 
\[
\min \left\{ F^{o}\left( u,v\right) -2\dint_{\Omega }f.udx\mid u\in
H_{\Gamma _{1}}^{1}\left( \Omega ,\mathbf{R}^{3}\right) ,\text{ }v\in
V\right\} .
\]%
Moreover, the convergence of the linked energies : $\lim_{\varepsilon
\rightarrow 0}F^{\varepsilon }\left( u^{\varepsilon }\right) =F^{o}\left(
u^{o},v^{o}\right) $ holds true.
\end{corollary}

\begin{remark}
\label{3.3}In the expression of the limit functional $F^{o}$, the term $\pi
E_{o}\int_{\Omega }\left( e_{33}\left( v\right) \right) ^{2}dx$ can be
interpreted as the \textquotedblright pure influence\textquotedblright\ of
the fibers, due to their longitudinal repartition, on the asymptotic
behaviour. The term $2\pi \gamma \int_{\Omega }\left( v-u\right) ^{t}A\left(
v-u\right) dx$ can be interpreted as the mixed influence of the fibers and
of the elastic material (for example, shearing effect of the fibers on the
material, for the term $2\pi \gamma \mu \int_{\Omega }\left(
v_{3}-u_{3}\right) ^{2}dx$).
\end{remark}

\begin{proof}[\textbf{of Theorem \protect\ref{3.1}}]
This proof will be decomposed in two main parts, corresponding to the
verification of the two assertions of the epi-convergence.\ As a first step,
let us verify : \textit{For every }$u$\textit{\ in }$H_{\Gamma
_{1}}^{1}\left( \Omega ,\mathbf{R}^{3}\right) $\textit{\ and for every }$v$%
\textit{\ in }$V$\textit{, there exists a sequence }$\left( u_{\varepsilon
}^{o}\right) _{\varepsilon }$\textit{\ of elements of }$H_{\Gamma
_{1}}^{1}\left( \Omega ,\mathbf{R}^{3}\right) $\textit{\ converging to }$%
\left( u,v\right) $\textit{\ in the topology }$\tau $\textit{\ and such that
: }$\lim \sup_{\varepsilon \rightarrow 0}F^{\varepsilon }\left(
u_{\varepsilon }^{o}\right) \leq F^{o}\left( u,v\right) .$

Let us first choose any element $u$ of $C^{1}\left( \overline{\Omega },%
\mathbf{R}^{3}\right) \cap H_{\Gamma _{1}}^{1}\left( \Omega ,\mathbf{R}%
^{3}\right) $ and any element $v$ of $C^{2}\left( \overline{\Omega },\mathbf{%
R}^{3}\right) \cap V$. For every $k=\left( k_{1},k_{2}\right) $, we define
the function $\mathcal{R}_{\varepsilon }\left( v\right) $ in $B_{\varepsilon
}^{k}$ by its three components as follows: 
\[
\left\{ 
\begin{array}{rcl}
\left( \mathcal{R}_{\varepsilon }\left( v\right) \right) _{\alpha }\left(
x_{1},x_{2},x_{3}\right)  & = & v_{\alpha }\left( k_{1}\varepsilon
,k_{2}\varepsilon ,x_{3}\right)  \\ 
&  & \qquad -\dfrac{\lambda ^{\varepsilon }}{2\left( \mu ^{\varepsilon
}+\lambda ^{\varepsilon }\right) }\left( x_{\alpha }-k_{\alpha }\varepsilon
\right) \dfrac{\partial v_{3}}{\partial x_{3}}\left( k_{1}\varepsilon
,k_{2}\varepsilon ,x_{3}\right)  \\ 
\left( \mathcal{R}_{\varepsilon }\left( v\right) \right) _{3}\left(
x_{1},x_{2},x_{3}\right)  & = & v_{3}\left( k_{1}\varepsilon
,k_{2}\varepsilon ,x_{3}\right) -\left( x_{1}-k_{1}\varepsilon \right) 
\dfrac{\partial v_{1}}{\partial x_{3}}\left( k_{1}\varepsilon
,k_{2}\varepsilon ,x_{3}\right)  \\ 
&  & \qquad -\left( x_{2}-k_{2}\varepsilon \right) \dfrac{\partial v_{2}}{%
\partial x_{3}}\left( k_{1}\varepsilon ,k_{2}\varepsilon ,x_{3}\right) .%
\end{array}%
\right. 
\]
\end{proof}

Let us choose some smooth function $\psi _{\varepsilon }$ identically equal
to 1 (resp. to 0) in $\Omega \setminus \overline{\Sigma _{2\varepsilon }}$
(resp.\ in $\Sigma _{\varepsilon }$), with : $\Sigma _{\varepsilon }=\left\{
x\in \Omega \mid d\left( x,\Gamma _{1}\right) <\varepsilon \right\} $. We
define: 
\begin{equation}
\begin{array}{ccl}
u_{\varepsilon }^{o} & = & \left( 1-\psi _{\varepsilon }\right) u+\psi
_{\varepsilon }\left( \left( e_{m}-z_{\varepsilon }^{m}\right)
u_{m}+z_{\varepsilon }^{m}\left( \mathcal{R}_{\varepsilon }\left( v\right)
\right) _{m}\right)  \\ 
& = & u-\psi _{\varepsilon }z_{\varepsilon }^{m}\left( u_{m}-\left( \mathcal{%
R}_{\varepsilon }\left( v\right) \right) _{m}\right) ,%
\end{array}
\label{uoeps}
\end{equation}%
where $u_{m}$ and $\left( \mathcal{R}_{\varepsilon }\left( v\right) \right)
_{m}$ are the $m$-th components of $u$ and $\mathcal{R}_{\varepsilon }\left(
v\right) $ in the canonical basis $\left( e_{m}\right) _{m=1,2,3}$ of $%
\mathbf{R}^{3}$ and $z_{\varepsilon }^{m}$ is defined in (\ref{zepsm}). One
has the following estimates.

\begin{lemma}
\label{3.4}

\begin{enumerate}
\item There exists some positive constant $C$ independant of $\varepsilon $
such that%
\[
\begin{array}{rcll}
\left\vert u_{\varepsilon }^{o}\right\vert \left( x\right)  & \leq  & C & 
\forall x\in \Omega  \\ 
\left\vert \nabla \mathcal{R}_{\varepsilon }\left( v\right) \right\vert
\left( x\right)  & \leq  & C & \forall x\in B_{\varepsilon } \\ 
\left\vert \mathcal{R}_{\varepsilon }\left( v\right) -v\right\vert \left(
x\right)  & \leq  & Cr_{\varepsilon } & \forall x\in T_{\varepsilon } \\ 
\left\vert \mathcal{R}_{\varepsilon }\left( v\right) -v\right\vert \left(
x\right)  & \leq  & Cs_{\varepsilon } & \forall x\in B_{\varepsilon }.%
\end{array}%
\]

\item $u_{\varepsilon }^{o}$ belongs to $H_{\Gamma _{1}}^{1}\left( \Omega ,%
\mathbf{R}^{3}\right) $, $\left( u_{\varepsilon }^{o}\right) _{\varepsilon }$
converges to $\left( u,v\right) $ in the above defined topology $\tau $.
\end{enumerate}
\end{lemma}

\begin{proof}
\textit{1}. Because $v$ belongs to $L^{\infty }\left( \Omega ,\mathbf{R}%
^{3}\right) $, together with its first order derivatives, we get, in every $%
B_{\varepsilon }^{k}$ : $\left\vert \mathcal{R}_{\varepsilon }\left(
v\right) \right\vert \leq C$ and $\left\vert \nabla \mathcal{R}_{\varepsilon
}\left( v\right) \right\vert \leq C^{\prime }$, where $C$ and $C^{\prime }$
are positive constants. Using Lemma \ref{2.2}, we get : $\left\vert
u_{\varepsilon }^{o}\right\vert \leq C$, in $\Omega $. One has, for every $%
k=\left( k_{1},k_{2}\right) $ 
\[
\begin{array}{ccl}
\left\vert \left( \mathcal{R}_{\varepsilon }\left( v\right) -v\right)
_{\alpha }\right\vert _{\mid T_{\varepsilon }^{k}} & \leq  & \left\vert
v_{\alpha }\left( k_{1}\varepsilon ,k_{2}\varepsilon ,x_{3}\right)
-v_{\alpha }\left( x_{1},x_{2},x_{3}\right) \right\vert  \\ 
&  & \qquad +\dfrac{\lambda ^{\varepsilon }}{2\left( \mu ^{\varepsilon
}+\lambda ^{\varepsilon }\right) }\left\vert \left( x_{\alpha }-k_{\alpha
}\varepsilon \right) \dfrac{\partial v_{3}}{\partial x_{3}}\left(
k_{1}\varepsilon ,k_{2}\varepsilon ,x_{3}\right) \right\vert  \\ 
& \leq  & Cr_{\varepsilon },%
\end{array}%
\]%
because $v$ belongs to $C^{1}\left( \overline{\Omega },\mathbf{R}^{3}\right) 
$ and using the hypotheses on $\lambda ^{\varepsilon }$ and $\mu
^{\varepsilon }$. Similarly, we have : $\left\vert \left( \mathcal{R}%
_{\varepsilon }\left( v\right) -v\right) _{3}\right\vert _{\mid
T_{\varepsilon }^{k}}\leq Cr_{\varepsilon }$, and : $\left\vert \mathcal{R}%
_{\varepsilon }\left( v\right) -v\right\vert _{\mid B_{\varepsilon
}^{k}}\leq Cs_{\varepsilon }$, for every $k$.

\textit{2}. Observe that $u_{\varepsilon }^{o}$ belongs to $H_{\Gamma
_{1}}^{1}\left( \Omega ,\mathbf{R}^{3}\right) $ because $u$ vanishes on $%
\Gamma _{1}$ and $\psi _{\varepsilon }$ also vanishes on $\Gamma _{1}$.
Furthermore, there exists some constant $C_{m}$ such that one has in $%
B_{\varepsilon }$ 
\begin{equation}
\begin{array}{ccl}
\left\vert \nabla u_{\varepsilon }^{o}\right\vert  & \leq  & \left\vert
\nabla u_{m}\left( e_{m}-z_{\varepsilon }^{m}\right) +z_{\varepsilon
}^{m}\nabla \left( \mathcal{R}_{\varepsilon }\left( v\right) \right)
_{m}+\left( \left( \mathcal{R}_{\varepsilon }\left( v\right) \right)
_{m}-u_{m}\right) \nabla z_{\varepsilon }^{m}\right\vert  \\ 
& \leq  & C_{m}\left( \left\vert \nabla u_{m}\right\vert +\varepsilon
\left\vert \nabla z_{\varepsilon }^{m}\right\vert +\left\vert \nabla
z_{\varepsilon }^{m}\right\vert \left\vert v_{m}-u_{m}\right\vert \right) ,%
\end{array}
\label{estim2}
\end{equation}%
for some constant $C_{m}$, thanks to the preceding estimates. We then
compute 
\begin{equation}
\dint_{\Omega }\left\vert \nabla u_{\varepsilon }^{o}\right\vert
^{2}dx=\dint_{\Omega \backslash \overline{B_{\varepsilon }}}\left\vert
\nabla u_{\varepsilon }^{o}\right\vert ^{2}dx+\dint_{B_{\varepsilon
}}\left\vert \nabla u_{\varepsilon }^{o}\right\vert ^{2}dx  \label{estim3}
\end{equation}

Thanks to (\ref{estim2}) and to Lemma \ref{2.3} one has 
\[
\begin{array}{ccl}
\dint_{B_{\varepsilon }}\left\vert \nabla u_{\varepsilon }^{o}\right\vert
^{2}dx & \leq  & C_{m}^{\prime }\left( \dint_{B_{\varepsilon }}\left\vert
\nabla u_{m}\right\vert ^{2}dx+\varepsilon \dint_{B_{\varepsilon
}}\left\vert \nabla z_{\varepsilon }^{m}\right\vert
^{2}dx+\dint_{B_{\varepsilon }}\left\vert v_{m}-u_{m}\right\vert
^{2}\left\vert \nabla z_{\varepsilon }^{m}\right\vert ^{2}dx\right)  \\ 
& \leq  & C,%
\end{array}%
\]%
where $C$ is some positive constant independant of $\varepsilon $.
Furthermore, because $z_{\varepsilon }^{m}$ outside $B_{\varepsilon }$ 
\[
\dint_{\Omega \backslash \overline{B_{\varepsilon }}}\left\vert \nabla
u_{\varepsilon }^{o}\right\vert ^{2}dx\underset{\varepsilon \rightarrow 0}{%
\rightarrow }\dint\nolimits_{\Omega }\left\vert \nabla u\right\vert ^{2}dx.
\]

This proves that $\left( u_{\varepsilon }^{o}\right) _{\varepsilon }$
converges to $u$ in the weak topology of $H^{1}\left( \Omega ,\mathbf{R}%
^{3}\right) $. Let $\varphi $ be any element of $C_{0}^{1}\left( \mathbf{R}%
^{3},\mathbf{R}^{3}\right) $. We have, because : $\left( z_{\varepsilon
}^{m}\right) _{\mid T_{\varepsilon }}=e_{m}$%
\[
\begin{array}{ccl}
\dint\nolimits_{\Omega }\varphi R^{\varepsilon }\left( u_{\varepsilon
}^{o}\right) dx & = & \dfrac{\left\vert \Omega \right\vert }{\left\vert
T_{\varepsilon }\right\vert }\dint\nolimits_{T_{\varepsilon }}\varphi
u_{\varepsilon }^{o}dx \\ 
& = & \dfrac{\left\vert \Omega \right\vert }{\left\vert T_{\varepsilon
}\right\vert }\dint\nolimits_{T_{\varepsilon }}\varphi \mathcal{R}%
_{\varepsilon }\left( v\right) dx \\ 
& = & \dfrac{\left\vert \Omega \right\vert \left\vert T_{\varepsilon
}^{k}\cap \omega \right\vert }{\left\vert T_{\varepsilon }\right\vert
\varepsilon ^{2}}\dsum\limits_{k}\varepsilon ^{2}\dint_{0}^{L}\varphi \left(
k_{1}\varepsilon ,k_{2}\varepsilon ,x_{3}\right) v\left( k_{1}\varepsilon
,k_{2}\varepsilon ,x_{3}\right) dx_{3}+o_{\varepsilon },%
\end{array}%
\]%
$\varphi $ and $v$ being continuously differentiable and $\left\vert
T_{\varepsilon }^{k}\cap \omega \right\vert $ being independant of $k$. We
have, thanks to the smoothness of $\varphi $ and $v$ 
\[
\underset{\varepsilon \rightarrow 0}{\lim }\dsum\limits_{k}\varepsilon
^{2}\dint_{0}^{L}\varphi \left( k_{1}\varepsilon ,k_{2}\varepsilon
,x_{3}\right) v\left( k_{1}\varepsilon ,k_{2}\varepsilon ,x_{3}\right)
dx_{3}=\dint_{\Omega }\varphi vdx
\]%
and we observe that : $\lim_{\varepsilon \rightarrow 0}\left( \left\vert
\Omega \right\vert \left\vert T_{\varepsilon }^{k}\cap \omega \right\vert
\right) /\left( \left\vert T_{\varepsilon }\right\vert \varepsilon
^{2}\right) =1$. This proves that the sequence $\left( u_{\varepsilon
}^{o}\right) _{\varepsilon }$ converges to $\left( u,v\right) $ in the above
defined topology $\tau $.\qquad $\square $
\end{proof}

For every $u$ in $C^{1}\left( \overline{\Omega },\mathbf{R}^{3}\right) \cap
H_{\Gamma _{1}}^{1}\left( \Omega ,\mathbf{R}^{3}\right) $ and every $v$ in $%
C^{1}\left( \overline{\Omega },\mathbf{R}^{3}\right) $, we compute 
\begin{equation}
\begin{array}{r}
F^{\varepsilon }\left( u_{\varepsilon }^{o}\right) =\dint_{\Omega \setminus 
\overline{C_{\varepsilon }\cup T_{\varepsilon }}}\sigma _{ij}\left( u\right)
e_{ij}\left( u\right) dx+\dint_{C_{\varepsilon }}\sigma _{ij}\left(
u_{\varepsilon }^{o}\right) e_{ij}\left( u_{\varepsilon }^{o}\right)
dx\qquad  \\ 
+\dint_{T_{\varepsilon }}\sigma _{ij}^{\varepsilon }\left( \mathcal{R}%
_{\varepsilon }\left( v\right) \right) e_{ij}\left( \mathcal{R}_{\varepsilon
}\left( v\right) \right) dx.%
\end{array}
\label{Fepsuoe}
\end{equation}

Because the characteristic function of $\Omega \backslash \overline{%
C_\varepsilon \cup T_\varepsilon }$ converges to 1 in the strong topology of 
$L^2\left( \Omega \right) $, the first integral of (\ref{Fepsuoe})
immediately leads to 
\begin{equation}  \label{calc3}
\underset{\varepsilon \rightarrow 0}{\lim }\int_{\Omega \setminus \overline{%
C_\varepsilon \cup T_\varepsilon }}\sigma _{ij}\left( u\right) e_{ij}\left(
u\right) dx=\int_\Omega \sigma _{ij}\left( u\right) e_{ij}\left( u\right) dx.
\end{equation}

Let us study the second integral of (\ref{Fepsuoe}). One has, using the
definition (\ref{uoeps}) of the test-function $u_\varepsilon ^o$ 
\begin{equation}  \label{int1}
\begin{array}{l}
\dint_{C_\varepsilon }\sigma _{ij}\left( u_\varepsilon ^o\right)
e_{ij}\left( u_\varepsilon ^o\right) dx \\ 
\qquad =\dint_{C_\varepsilon }\sigma _{ij}\left( u\right) e_{ij}\left(
u\right) dx+2\dint_{C_\varepsilon }\sigma _{ij}\left( u\right) e_{ij}\left(
z_\varepsilon ^m\left( \left( \mathcal{R}_\varepsilon \left( v\right)
\right) _m-u_m\right) \right) dx\qquad \\ 
\hfill+\dint_{C_\varepsilon }\sigma _{ij}\left( z_\varepsilon ^m\left(
\left( \mathcal{R}_\varepsilon \left( v\right) \right) _m-u_m\right) \right)
e_{ij}\left( z_\varepsilon ^l\left( \left( \mathcal{R}_\varepsilon \left(
v\right) \right) _l-u_l\right) \right) dx.%
\end{array}%
\end{equation}

The second integral of the right hand side of (\ref{int1}) converges to 0,
because $\left( z_{\varepsilon }^{m}\right) _{\varepsilon }$ converges to $0$
in the weak topology of $H_{\Gamma _{1}}^{1}\left( \Omega ,\mathbf{R}%
^{3}\right) $ and thanks to the estimates of Lemma \ref{3.3}. The third
integral of this right hand side of (\ref{int1}) can be computed as 
\begin{equation}
\begin{array}{l}
\dint_{C_{\varepsilon }}\sigma _{ij}\left( z_{\varepsilon }^{m}\left(
v_{m}-u_{m}\right) \right) e_{ij}\left( z_{\varepsilon }^{l}\left(
v_{l}-u_{l}\right) \right) dx \\ 
\qquad +2\dint_{C_{\varepsilon }}\sigma _{ij}\left( z_{\varepsilon
}^{m}\left( \left( \mathcal{R}_{\varepsilon }\left( v\right) \right)
_{m}-u_{m}\right) \right) e_{ij}\left( z_{\varepsilon }^{l}\left(
v_{l}-u_{l}\right) \right) dx \\ 
\qquad +\dint_{C_{\varepsilon }}\sigma _{ij}\left( z_{\varepsilon
}^{m}\left( \left( \mathcal{R}_{\varepsilon }\left( v\right) \right)
_{m}-v_{m}\right) \right) e_{ij}\left( z_{\varepsilon }^{l}\left( \left( 
\mathcal{R}_{\varepsilon }\left( v\right) \right) _{l}-v_{l}\right) \right)
dx.%
\end{array}
\label{int2}
\end{equation}

Thanks to Lemmas \ref{2.3} and \ref{3.4}, the two last integrals of (\ref%
{int2}) converge to $0$ and the first integral of (\ref{int2}) is equal to 
\[
\dint\nolimits_{\Omega }\sigma _{ij}\left( z_{\varepsilon }^{m}\right)
e_{ij}\left( z_{\varepsilon }^{l}\right) \left( v_{m}-u_{m}\right) \left(
v_{l}-u_{l}\right) dx+o_{\varepsilon },
\]%
with $\lim_{\varepsilon \rightarrow 0}o_{\varepsilon }=0$, because $\left(
z_{\varepsilon }^{m}\right) _{\varepsilon }$ converges to $0$ in the weak
topology of $H_{\Gamma _{1}}^{1}\left( \Omega ,\mathbf{R}^{3}\right) $. One
deduces from Lemma \ref{2.3} and the smoothness of $u$ and $v$ that 
\begin{equation}
\underset{\varepsilon \rightarrow 0}{\lim }\dint\nolimits_{\Omega }\sigma
_{ij}\left( z_{\varepsilon }^{m}\right) e_{ij}\left( z_{\varepsilon
}^{l}\right) \left( v_{m}-u_{m}\right) \left( v_{l}-u_{l}\right) dx=2\pi
\gamma \dint_{\Omega }\left( v-u\right) ^{t}A\left( v-u\right) dx.
\label{calc1}
\end{equation}

In order to study the third integral of (\ref{Fepsuoe}), one observes that
the above expression of $\mathcal{R}_\varepsilon \left( v\right) $ implies

\[
\begin{array}{rcl}
Tr\left( e\left( \mathcal{R}_\varepsilon \left( v\right) \right) \right) & =
& \dfrac{\mu ^\varepsilon }{\mu ^\varepsilon +\lambda ^\varepsilon }\dfrac{%
\partial v_3}{\partial x_3}\left( k_1\varepsilon ,k_2\varepsilon ,x_3\right)
-\left( x_\alpha -k_\alpha \varepsilon \right) \dfrac{\partial ^2v_\alpha }{%
\partial x_3^2}\left( k_1\varepsilon ,k_2\varepsilon ,x_3\right) \\ 
\sigma _{11}^\varepsilon \left( \mathcal{R}_\varepsilon \left( v\right)
\right) & = & -\lambda ^\varepsilon \left( x_\alpha -k_\alpha \varepsilon
\right) \dfrac{\partial ^2v_\alpha }{\partial x_3^2}\left( k_1\varepsilon
,k_2\varepsilon ,x_3\right) \\ 
\sigma _{22}^\varepsilon \left( \mathcal{R}_\varepsilon \left( v\right)
\right) & = & -\lambda ^\varepsilon \left( x_\alpha -k_\alpha \varepsilon
\right) \dfrac{\partial ^2v_\alpha }{\partial x_3^2}\left( k_1\varepsilon
,k_2\varepsilon ,x_3\right) \\ 
\sigma _{12}^\varepsilon \left( \mathcal{R}_\varepsilon \left( v\right)
\right) & = & 0 \\ 
\sigma _{33}^\varepsilon \left( \mathcal{R}_\varepsilon \left( v\right)
\right) & = & \mu ^\varepsilon \dfrac{2\mu ^\varepsilon +3\lambda
^\varepsilon }{\mu ^\varepsilon +\lambda ^\varepsilon }\dfrac{\partial v_3}{%
\partial x_3}\left( k_1\varepsilon ,k_2\varepsilon ,x_3\right) \\ 
&  & \qquad -\left( 2\mu ^\varepsilon +\lambda ^\varepsilon \right) \left(
x_\alpha -k_\alpha \varepsilon \right) \dfrac{\partial ^2v_\alpha }{\partial
x_3^2}\left( k_1\varepsilon ,k_2\varepsilon ,x_3\right) \\ 
\sigma _{\alpha 3}^\varepsilon \left( \mathcal{R}_\varepsilon \left(
v\right) \right) & = & -\mu ^\varepsilon \left( x_\alpha -k_\alpha
\varepsilon \right) \dfrac{\partial ^2v_\alpha }{\partial x_3^2}\left(
k_1\varepsilon ,k_2\varepsilon ,x_3\right) .%
\end{array}
\]

One easily proves that all the terms of the third integral of (\ref{Fepsuoe}%
) converge to 0 except the following one 
\[
\begin{array}{l}
\dint_{T_\varepsilon }\sigma _{33}^\varepsilon \left( \mathcal{R}%
_\varepsilon \left( v\right) \right) e_{33}\left( \mathcal{R}_\varepsilon
\left( v\right) \right) dx \\ 
\qquad 
\begin{array}{cl}
= & \dfrac{\pi \mu ^\varepsilon \left( r_\varepsilon \right) ^2}{\varepsilon
^2}\dfrac{2\mu ^\varepsilon +3\lambda ^\varepsilon }{\mu ^\varepsilon
+\lambda ^\varepsilon }\dsum\limits_k\varepsilon ^2\dint\nolimits_0^L\left( 
\dfrac{\partial v_3}{\partial x_3}\right) ^2\left( k_1\varepsilon
,k_2\varepsilon ,x_3\right) dx_3+o_\varepsilon \\ 
\underset{\varepsilon \rightarrow 0}{\rightarrow } & \pi E_o\dint_\Omega
\left( e_{33}\left( v\right) \right) ^2dx,%
\end{array}%
\end{array}
\]
with the above definition of $E_o$. Thus, we get, for this third integral of
(\ref{Fepsuoe}) 
\begin{equation}  \label{calc2}
\underset{\varepsilon \rightarrow 0}{\lim }\dint\nolimits_{T_\varepsilon
}\sigma _{ij}^\varepsilon \left( \mathcal{R}_\varepsilon \left( v\right)
\right) e_{ij}\left( \mathcal{R}_\varepsilon \left( v\right) \right) dx=\pi
E_o\dint_\Omega \left( e_{33}\left( v\right) \right) ^2dx.
\end{equation}

From (\ref{calc3}), (\ref{calc1}) and (\ref{calc2}), we thus derive : $\lim
_{\varepsilon \rightarrow 0}F^\varepsilon \left( u_\varepsilon ^o\right)
=F^o\left( u,v\right) .$

We conclude the verification of this first assertion, using a density
argument and the diagonalization argument contained in \cite[Corollary 1.18]%
{Att}. Indeed, for every $u$ in $H_{\Gamma _{1}}^{1}\left( \Omega ,\mathbf{R}%
^{3}\right) ,$ there exists a sequence $\left( u^{n},v^{n}\right) _{n}$ in $%
\left( C^{1}\left( \overline{\Omega },\mathbf{R}^{3}\right) \cap H_{\Gamma
_{1}}^{1}\left( \Omega ,\mathbf{R}^{3}\right) \right) \times \left(
C^{2}\left( \overline{\Omega },\mathbf{R}^{3}\right) \cap V\right) $
converging to $\left( u,v\right) $ in the strong topology of the space $%
H^{1}\left( \Omega ,\mathbf{R}^{3}\right) \times V$. Thanks to Lemma \ref%
{3.4}, $\left( \left( u^{n}\right) _{\varepsilon }^{o}\right) _{\varepsilon }
$ converges to $\left( u^{n},v^{n}\right) $ in the topology $\tau $ and 
\[
\underset{n\rightarrow +\infty }{\lim }\ \underset{\varepsilon \rightarrow 0}%
{\lim }F^{\varepsilon }\left( \left( u^{n}\right) _{\varepsilon }^{o}\right)
=\ \underset{n\rightarrow +\infty }{\lim }F^{o}\left( u^{n},v^{n}\right)
=F^{o}\left( u,v\right) .
\]

The space $H^{1}\left( \Omega ,\mathbf{R}^{3}\right) \times L^{1}\left(
\Omega ,\mathbf{R}^{3}\right) $ is metrizable for the topology $\tau $. One
deduces from \cite[Corollary 1.18]{Att}, the existence of a subsequence $%
\left( \left( u^{n\left( \varepsilon \right) }\right) _{\varepsilon
}^{o}\right) _{\varepsilon }$ converging to $u$ in the weak topology of $%
H_{\Gamma _{1}}^{1}\left( \Omega ,\mathbf{R}^{3}\right) $, such that $\left(
R^{\varepsilon }\left( v^{n\left( \varepsilon \right) }\right) \right)
_{\varepsilon }$ converges to $v$ in the weak$^{\ast }$ topology of $%
L^{1}\left( \Omega ,\mathbf{R}^{3}\right) $ and : $\lim \sup_{\varepsilon
\rightarrow 0}F^{\varepsilon }\left( \left( u^{n\left( \varepsilon \right)
}\right) _{\varepsilon }^{o}\right) \leq F^{o}\left( u,v\right) $. This ends
the verification of the first assertion.

Let us now prove the second assertion of the epi-convergence, that is : 
\textit{For every sequence }$\left( u_{\varepsilon }\right) _{\varepsilon }$%
\textit{\ of elements of }$H_{\Gamma _{1}}^{1}\left( \Omega ,\mathbf{R}%
^{3}\right) $\textit{, converging to }$\left( u,v\right) $\textit{\ in the
topology }$\tau $\textit{, then }$v$\textit{\ belongs to }$V$\textit{,
satisfies : }$v=0$\textit{, on }$\Gamma _{1}$\textit{, and : }$\lim
\inf_{\varepsilon \rightarrow 0}F^{\varepsilon }\left( u_{\varepsilon
}\right) \geq F^{o}\left( u,v\right) $.

Let $\left( u^{n}\right) _{n}$ be any sequence of smooth functions in $%
C^{1}\left( \overline{\Omega },\mathbf{R}^{3}\right) \cap H_{\Gamma
_{1}}^{1}\left( \Omega ,\mathbf{R}^{3}\right) $ converging to $u$ in the
strong topology of $H^{1}\left( \Omega ,\mathbf{R}^{3}\right) $ and $\left(
v^{n}\right) _{n}$ be any sequence of smooth functions in $C^{2}\left( 
\overline{\Omega },\mathbf{R}^{3}\right) \cap V$ converging to $v$ in the
strong topology of $V$. Let us suppose that $\sup_{\varepsilon
}F^{\varepsilon }\left( u_{\varepsilon }\right) <+\infty $, otherwise the
assertion is trivially satisfied. Under these hypotheses, one proves

\begin{lemma}
\label{3.5}$\left( u_{\varepsilon }\right) _{\varepsilon }$ is bounded in $%
H_{\Gamma _{1}}^{1}\left( \Omega ,\mathbf{R}^{3}\right) $ and the sequence $%
\left( R^{\varepsilon }\left( u_{\varepsilon }\right) \right) _{\varepsilon }
$ converges in the weak$^{\ast }$ topology of $L^{1}\left( \Omega ,\mathbf{R}%
^{3}\right) $ to some $v$ belonging to $V$.
\end{lemma}

\begin{proof}
We use some argument similar to \cite[Lemme A1]{Bel-Bou}, defining:%
\[
\Phi _{\varepsilon }=e_{33}\left( u_{\varepsilon }\right) \text{, }\delta
_{\varepsilon }=\dfrac{\left\vert \Omega \right\vert }{\left\vert
T_{\varepsilon }\right\vert }\mathbf{1}_{T_{\varepsilon }}dx\text{, }\delta =%
\mathbf{1}_{\Omega }dx.
\]

$\delta _{\varepsilon }$ and $\delta $ are two bounded Radon measures such
that $\left( \delta _{\varepsilon }\right) _{\varepsilon }$ converges weakly
to $\delta $ in the sense of measures. We then compute 
\[
\begin{array}{rcl}
\dint_{\mathbf{R}^{3}}\left\vert \Phi _{\varepsilon }\right\vert \delta
_{\varepsilon } & \leq  & \left( \dint_{\mathbf{R}^{3}}\left\vert \Phi
_{\varepsilon }\right\vert ^{2}\delta _{\varepsilon }\right) ^{1/2}\sqrt{%
\left\vert T_{\varepsilon }\right\vert } \\ 
& \leq  & \dfrac{C}{\sqrt{\left\vert T_{\varepsilon }\right\vert }}\left(
\dint_{T_{\varepsilon }}\left\vert \Phi _{\varepsilon }\right\vert
^{2}dx\right) ^{1/2}\leq C\left( \underset{\varepsilon }{\sup }%
F^{\varepsilon }\left( u_{\varepsilon }\right) \right) ^{1/2}<+\infty ,%
\end{array}%
\]%
because $\left( \lambda ^{\varepsilon }\left\vert T_{\varepsilon
}\right\vert \right) _{\varepsilon }$ and $\left( \mu ^{\varepsilon
}\left\vert T_{\varepsilon }\right\vert \right) _{\varepsilon }$ have finite
limits. Hence, the sequence $\left( \Phi _{\varepsilon }\delta _{\varepsilon
}\right) _{\varepsilon }$ of measures has uniformly bounded variations. One
can extract some subsequence, still denoted by $\left( \Phi _{\varepsilon
}\delta _{\varepsilon }\right) _{\varepsilon }$, which converges to some
measure $\Phi $. For every $\varphi $ in $C_{c}^{o}\left( \mathbf{R}%
^{3}\right) $, we write Fenchel's inequality 
\[
\dint_{\mathbf{R}^{3}}\left\vert \Phi _{\varepsilon }\right\vert ^{2}\delta
_{\varepsilon }\geq 2\dint_{\mathbf{R}^{3}}\Phi _{\varepsilon }\varphi
\delta _{\varepsilon }-\dint_{\mathbf{R}^{3}}\varphi ^{2}\delta
_{\varepsilon },
\]%
which implies 
\[
\underset{\varepsilon \rightarrow 0}{\lim \inf }\dint_{\mathbf{R}%
^{3}}\left\vert \Phi _{\varepsilon }\right\vert ^{2}\delta _{\varepsilon
}\geq 2\left\langle \Phi ,\varphi \right\rangle -\dint_{\mathbf{R}%
^{3}}\varphi ^{2}\delta ,
\]%
where $\left\langle .,.\right\rangle $ means the duality product between
measures and functions, from which we deduce that : $\sup \left\{
\left\langle \Phi ,\varphi \right\rangle \mid \varphi \in C_{c}^{o}\left( 
\mathbf{R}^{3}\right) \text{, }\left\Vert \varphi \right\Vert _{L^{2}\left(
\Omega \right) }\leq 1\right\} <+\infty .$ Riesz's representation theorem
implies the existence of some $\chi $ in $L_{\delta }^{2}\left( \Omega
\right) $ such that for every $\varphi $ in $C_{c}^{o}\left( \mathbf{R}%
^{3}\right) $ : $\left\langle \Phi ,\varphi \right\rangle =\int_{\mathbf{R}%
^{3}}\chi \varphi \delta =\int_{\Omega }\chi \varphi dx$. For every $\varphi 
$ in $C_{0}^{1}\left( \Omega \right) $, one has 
\[
\begin{array}{ccl}
\underset{\varepsilon \rightarrow 0}{\lim }\dfrac{\left\vert \Omega
\right\vert }{\left\vert T_{\varepsilon }\right\vert }\dint\nolimits_{T_{%
\varepsilon }}e_{33}\left( u_{\varepsilon }\right) \varphi dx & = & 
\dint\nolimits_{\Omega }\chi \varphi dx \\ 
& = & -\underset{\varepsilon \rightarrow 0}{\lim }\dfrac{\left\vert \Omega
\right\vert }{\left\vert T_{\varepsilon }\right\vert }\dint\nolimits_{T_{%
\varepsilon }}\dfrac{\partial \varphi }{\partial x_{3}}\left( u_{\varepsilon
}\right) _{3}dx \\ 
& \underset{\varepsilon \rightarrow 0}{\rightarrow } & -\dint\nolimits_{%
\Omega }\dfrac{\partial \varphi }{\partial x_{3}}v_{3}dx=\dint\nolimits_{%
\Omega }\varphi e_{33}\left( v\right) dx.%
\end{array}%
\]

We thus get : $\int\nolimits_{\Omega }\left( \chi \varphi -\varphi
e_{33}\left( v\right) \right) dx=0$, which implies that $e_{33}\left(
v\right) $ ($=\chi $) belongs to $L^{2}\left( \Omega \right) $.

In order to prove that $v_{i}$ belongs to $L^{2}\left( \Omega \right) $, for 
$i=1,2,3$, we repeat the above argument with $\Phi _{\varepsilon ,i}=\left(
u_{\varepsilon }\right) _{i}$ instead of $\Phi _{\varepsilon }=e_{33}\left(
u_{\varepsilon }\right) $ and we use the estimates of Lemma \ref{1.1}\ 3.

In order to prove that $v_{3}$ is equal to 0 on $\Gamma _{1}$, let us take
any function $\varphi $ in $C^{1}\left( \overline{\Omega }\right) $ taking
the form: $\varphi \left( x\right) =\theta \left( x_{1},x_{2}\right) \psi
\left( x_{3}\right) $, with $\psi \left( 0\right) =1$, $\psi \left( L\right)
=0$, $\theta $ in $C^{\infty }\left( \omega \right) $. We first compute%
\[
\begin{array}{l}
\dint_{\Omega }\dfrac{\partial v_{3}}{\partial x_{3}}\varphi dx \\ 
\qquad 
\begin{array}{ll}
= & -\dint_{\Omega }\dfrac{\partial \varphi }{\partial x_{3}}v_{3}dx+\ 
\underset{\varepsilon \rightarrow 0}{\lim }\dfrac{\left\vert \Omega
\right\vert }{\left\vert T_{\varepsilon }\right\vert }\dint_{T_{\varepsilon
}}\left( 
\begin{array}{l}
\left( \varphi \left( u_{\varepsilon }\right) _{3}\right) \left(
x_{1},x_{2},L\right)  \\ 
-\left( \varphi \left( u_{\varepsilon }\right) _{3}\right) \left(
x_{1},x_{2},0\right) 
\end{array}%
\right) dx_{1}dx_{2} \\ 
= & -\dint_{\Omega }\dfrac{\partial \varphi }{\partial x_{3}}v_{3}dx,%
\end{array}%
\end{array}%
\]%
thanks to the boundary conditions verified by $\varphi $ and $u_{\varepsilon
}$. Moreover, using Green's formula, we get%
\[
\dint_{\Omega }\dfrac{\partial v_{3}}{\partial x_{3}}\varphi
dx=-\dint_{\Omega }\dfrac{\partial \varphi }{\partial x_{3}}%
v_{3}dx+\dint_{\omega }\theta \left( x_{1},x_{2}\right) v_{3}\left(
x_{1},x_{2},0\right) dx_{1}dx_{2},
\]%
which implies%
\[
\dint_{\omega }\theta \left( x_{1},x_{2}\right) v_{3}\left(
x_{1},x_{2},0\right) dx_{1}dx_{2}=0\Rightarrow v_{3}\left(
x_{1},x_{2},0\right) =0.
\]

Thus $v$ belongs to $V$.\qquad $\square $
\end{proof}

In order to prove this second assertion, we write the subdifferential
inequality for the first term of $F^{\varepsilon }\left( u_{\varepsilon
}\right) $ 
\[
\begin{array}{r}
\dint_{\Omega \backslash \overline{C_{\varepsilon }\cup T_{\varepsilon }}%
}\sigma _{ij}\left( u_{\varepsilon }\right) e_{ij}\left( u_{\varepsilon
}\right) dx\geq \dint_{\Omega \backslash \overline{C_{\varepsilon }\cup
T_{\varepsilon }}}\sigma _{ij}\left( \left( u^{n}\right) _{\varepsilon
}^{o}\right) e_{ij}\left( \left( u^{n}\right) _{\varepsilon }^{o}\right)
dx\qquad  \\ 
+2\dint_{\Omega \backslash \overline{C_{\varepsilon }\cup T_{\varepsilon }}%
}\sigma _{ij}\left( u^{n}\right) e_{ij}\left( u_{\varepsilon }-\left(
u^{n}\right) _{\varepsilon }^{o}\right) dx,%
\end{array}%
\]%
where $\left( u^{n}\right) _{\varepsilon }^{o}$ is associated to $u^{n}$
through (\ref{uoeps}). The sequence $\left( \left( u^{n}\right)
_{\varepsilon }^{o}\right) _{\varepsilon }$ converges to $u^{n}$ in the weak
topology of $H^{1}\left( \Omega ,\mathbf{R}^{3}\right) $, thanks to Lemma %
\ref{3.4}, and coincides with $u_{n}$ in $\Omega \backslash \overline{%
C_{\varepsilon }\cup T_{\varepsilon }}$. Thus, $\left( e_{ij}\left(
u_{\varepsilon }-\left( u^{n}\right) _{\varepsilon }^{o}\right) \right)
_{\varepsilon }$ converges to $e_{ij}\left( u-u^{n}\right) $ in the weak
topology of $L^{2}\left( \Omega \right) $, for $i,j=1,2,3$. The sequence of
characteristic functions of $\Omega \backslash \overline{C_{\varepsilon
}\cup T_{\varepsilon }}$ converges to $1$ in the strong topology of $%
L^{2}\left( \Omega \right) $. This implies the following convergence 
\[
\begin{array}{r}
\underset{\varepsilon \rightarrow 0}{\lim \inf }\dint_{\Omega \backslash 
\overline{C_{\varepsilon }\cup T_{\varepsilon }}}\sigma _{ij}\left(
u_{\varepsilon }\right) e_{ij}\left( u_{\varepsilon }\right) dx\geq
\dint_{\Omega }\sigma _{ij}\left( u^{n}\right) e_{ij}\left( u^{n}\right)
dx\qquad  \\ 
+2\dint_{\Omega }\sigma _{ij}\left( u^{n}\right) e_{ij}\left( u-u^{n}\right)
dx.%
\end{array}%
\]

Letting $n$ increase to $+\infty $ we get, using the convergence of $\left(
u^n\right) _n$ to $u$ in the strong topology of $H^1\left( \Omega ,\mathbf{R}%
^3\right) $

\begin{equation}  \label{calc4}
\underset{\varepsilon \rightarrow 0}{\lim \inf }\dint_{\Omega \backslash 
\overline{C_\varepsilon \cup T_\varepsilon }}\sigma _{ij}\left(
u_\varepsilon \right) e_{ij}\left( u_\varepsilon \right) dx\geq \dint_\Omega
\sigma _{ij}\left( u\right) e_{ij}\left( u\right) dx.
\end{equation}

We then write the subdifferential inequality for the second term of $%
F^\varepsilon \left( u_\varepsilon \right) $

\[
\begin{array}{r}
\dint\nolimits_{C_\varepsilon }\sigma _{ij}\left( u_\varepsilon \right)
e_{ij}\left( u_\varepsilon \right) dx\geq \dint\nolimits_{C_\varepsilon
}\sigma _{ij}\left( \left( u^n\right) _\varepsilon ^o\right) e_{ij}\left(
\left( u^n\right) _\varepsilon ^o\right) dx\qquad \\ 
+2\dint\nolimits_{C_\varepsilon }\sigma _{ij}\left( \left( u^n\right)
_\varepsilon ^o\right) e_{ij}\left( u_\varepsilon -\left( u_\varepsilon
^n\right) _\varepsilon ^o\right) dx,%
\end{array}
\]
with

\[
\begin{array}{r}
2\dint\nolimits_{C_\varepsilon }\sigma _{ij}\left( \left( u^n\right)
_\varepsilon ^o\right) e_{ij}\left( u_\varepsilon -\left( u^n\right)
_\varepsilon ^o\right) dx=2\dint\nolimits_{C_\varepsilon }\sigma _{ij}\left(
u^n\right) e_{ij}\left( u_\varepsilon -\left( u^n\right) _\varepsilon
^o\right) dx\qquad \\ 
+2\dint\nolimits_{C_\varepsilon }\sigma _{ij}\left( z_\varepsilon ^m\left(
\left( \mathcal{R}_\varepsilon \left( v^n\right) \right) _m-\left(
u^n\right) _m\right) \right) e_{ij}\left( u_\varepsilon -\left( u^n\right)
_\varepsilon ^o\right) dx.%
\end{array}
\]

We immediately get : $\lim _{\varepsilon \rightarrow
0}\int\nolimits_{C_\varepsilon }\sigma _{ij}\left( u^n\right) e_{ij}\left(
u_\varepsilon -\left( u^n\right) _\varepsilon ^o\right) dx=0,$ because the
sequence $\left( e_{ij}\left( u_\varepsilon -\left( u^n\right) _\varepsilon
^o\right) \right) _\varepsilon $ converges to $e_{ij}\left( u-u^n\right) $
in the weak topology of $L^2\left( \Omega \right) $, for $i,j=1,2,3$ and the
sequence of characteristic functions of $C_\varepsilon $ converges to 0 in
the strong topology of $L^2\left( \Omega \right) $. The second term of the
last equality can be computed as 
\[
\begin{array}{l}
\dint\nolimits_{C_\varepsilon }\sigma _{ij}\left( z_\varepsilon ^m\left(
\left( \mathcal{R}_\varepsilon \left( v^n\right) \right) _m-\left(
u^n\right) _m\right) \right) e_{ij}\left( u_\varepsilon -\left( u^n\right)
_\varepsilon ^o\right) dx \\ 
\qquad =\dint\nolimits_{C_\varepsilon }\sigma _{ij}\left( z_\varepsilon
^m\right) \left( \left( \mathcal{R}_\varepsilon \left( v^n\right) \right)
_m-\left( u^n\right) _m\right) e_{ij}\left( u_\varepsilon -\left( u^n\right)
_\varepsilon ^o\right) dx \\ 
\qquad \qquad +\dint\nolimits_{C_\varepsilon }a_{ijst}\left( z_\varepsilon
^m\right) _s\dfrac{\partial \left( \left( \mathcal{R}_\varepsilon \left(
v^n\right) \right) _m-\left( u^n\right) _m\right) }{\partial x_t}%
e_{ij}\left( u_\varepsilon -\left( u^n\right) _\varepsilon ^o\right) dx,%
\end{array}
\]
writing : $\sigma _{ij}=a_{ijst}e_{st}$. We observe that 
\[
\underset{\varepsilon \rightarrow 0}{\lim }\dint\nolimits_{C_\varepsilon
}a_{ijst}\left( z_\varepsilon ^m\right) _s\dfrac{\partial \left( \left( 
\mathcal{R}_\varepsilon \left( v^n\right) \right) _m-\left( u^n\right)
_m\right) }{\partial x_t}e_{ij}\left( u_\varepsilon -\left( u^n\right)
_\varepsilon ^o\right) dx=0, 
\]
because $\left( z_\varepsilon ^m\right) _\varepsilon $ converges to $0$ in
the strong topology of $L^2\left( \Omega ,\mathbf{R}^3\right) $, $\left|
\nabla \left( \mathcal{R}_\varepsilon \left( v^n\right) -u^n\right) \right|
\leq C_n$, in $C_\varepsilon $, and $\left( e_{ij}\left( u_\varepsilon
-\left( u^n\right) _\varepsilon ^o\right) \right) _\varepsilon $ converges
to $e_{ij}\left( u-u^n\right) $ in the weak topology of $L^2\left( \Omega
\right) $, for $i,j=1,2,3$. Then, we compute, using the definition of $%
z_\varepsilon ^m$ 
\[
\begin{array}{l}
\dint\nolimits_{C_\varepsilon }\sigma _{ij}\left( z_\varepsilon ^m\right)
\left( \left( \mathcal{R}_\varepsilon \left( v^n\right) \right) _m-\left(
u^n\right) _m\right) e_{ij}\left( u_\varepsilon -\left( u^n\right)
_\varepsilon ^o\right) dx \\ 
\quad =-\dsum\limits_k\dint\nolimits_{C_\varepsilon ^k}\sigma _{ij}\left(
w_\varepsilon ^{mk}\right) e_{ij}\left( u_\varepsilon -\left( u^n\right)
_\varepsilon ^o\right) \left( \left( \mathcal{R}_\varepsilon \left(
v^n\right) \right) _m-\left( u^n\right) _m\right) \varphi _\varepsilon ^kdx
\\ 
\qquad +\dsum\limits_k\dint\nolimits_{C_\varepsilon ^k}a_{ijst}\left( \left( 
\mathcal{R}_\varepsilon \left( v^n\right) \right) _m-\left( u^n\right)
_m\right) \left( e_m-w_\varepsilon ^{mk}\right) _s\dfrac{\partial \varphi
_\varepsilon ^k}{\partial x_t}e_{ij}\left( u_\varepsilon -\left( u^n\right)
_\varepsilon ^o\right) dx.%
\end{array}
\]

But, for every $k$, one has, thanks to the definition (\ref{phiepsk}) of $%
\varphi _{\varepsilon }^{k}$ and using Lemmas \ref{2.2} and \ref{3.4}
assertion \textit{1}. 
\[
\begin{array}{r}
\left\vert \dint\nolimits_{C_{\varepsilon }^{k}}\left( \left( \mathcal{R}%
_{\varepsilon }\left( v^{n}\right) \right) _{m}-\left( u^{n}\right)
_{m}\right) \left( e_{m}-w_{\varepsilon }^{mk}\right) _{s}\dfrac{\partial
\varphi _{\varepsilon }^{k}}{\partial x_{t}}e_{ij}\left( u_{\varepsilon
}-\left( u^{n}\right) _{\varepsilon }^{o}\right) dx\right\vert \qquad  \\ 
\leq \dfrac{C_{n}\left\vert \ln s_{\varepsilon }\right\vert }{\varepsilon
^{2}\left\vert \ln r_{\varepsilon }\right\vert }\dint\nolimits_{C_{%
\varepsilon }^{k}\cap \left\{ s_{\varepsilon }/2<R_{\varepsilon
}^{k}<s_{\varepsilon }\right\} }R_{\varepsilon }^{k}\left\vert \nabla \left(
u_{\varepsilon }-\left( u^{n}\right) _{\varepsilon }^{o}\right) \right\vert
dx.%
\end{array}%
\]

This implies, because $\left( u_{\varepsilon }\right) _{\varepsilon }$ and $%
\left( \left( u^{n}\right) _{\varepsilon }^{o}\right) _{\varepsilon }$ are
bounded in $H^{1}\left( \Omega ,\mathbf{R}^{3}\right) $ 
\[
\begin{array}{r}
\underset{\varepsilon \rightarrow 0}{\lim \sup }\left\vert
\dsum\limits_{k}\dint\nolimits_{C_{\varepsilon }^{k}}a_{ijst}\left( \left( 
\mathcal{R}_{\varepsilon }\left( v^{n}\right) \right) _{m}-\left(
u^{n}\right) _{m}\right) \left( e_{m}-w_{\varepsilon }^{mk}\right) _{s}%
\dfrac{\partial \varphi _{\varepsilon }^{k}}{\partial x_{t}}e_{ij}\left(
u_{\varepsilon }-\left( u^{n}\right) _{\varepsilon }^{o}\right)
dx\right\vert \qquad  \\ 
\leq \ \underset{\varepsilon \rightarrow 0}{\lim \sup }\dfrac{%
C_{n}\left\vert \ln s_{\varepsilon }\right\vert }{\varepsilon \left\vert \ln
r_{\varepsilon }\right\vert }\left( \dint\nolimits_{\Omega }\left\vert
\nabla \left( u_{\varepsilon }-\left( u^{n}\right) _{\varepsilon
}^{o}\right) \right\vert ^{2}dx\right) ^{1/2}=0,%
\end{array}%
\]%
because $\gamma $ is finite and using the properties of $s_{\varepsilon }$.\
Similarly, we estimate, using Lemma \ref{2.2} 
\[
\begin{array}{r}
\left\vert \dsum\limits_{k}\dint\nolimits_{C_{\varepsilon }^{k}}\sigma
_{ij}\left( w_{\varepsilon }^{mk}\right) e_{ij}\left( u_{\varepsilon
}-\left( u^{n}\right) _{\varepsilon }^{o}\right) \left( \left( \mathcal{R}%
_{\varepsilon }\left( v^{n}\right) \right) _{m}-\left( v^{n}\right)
_{m}\right) \varphi _{\varepsilon }^{k}dx\right\vert \qquad  \\ 
\leq \dfrac{C_{n}\sqrt{s_{\varepsilon }}}{\left\vert \ln \left(
r_{\varepsilon }\right) \right\vert }\left( \dint\nolimits_{\Omega
}\left\vert \nabla \left( u_{\varepsilon }-\left( u^{n}\right) _{\varepsilon
}^{o}\right) \right\vert ^{2}dx\right) ^{1/2}\underset{\varepsilon
\rightarrow 0}{\rightarrow }0,%
\end{array}%
\]%
because $\gamma $ is finite.\ We then have to compute the limit of the
remaining term

\[
\begin{array}{l}
\dsum\limits_k\dint\nolimits_{C_\varepsilon ^k}\sigma _{ij}\left(
w_\varepsilon ^{mk}\right) e_{ij}\left( u_\varepsilon -\left( u^n\right)
_\varepsilon ^o\right) \left( \left( v^n\right) _m-\left( u^n\right)
_m\right) \varphi _\varepsilon ^kdx \\ 
\qquad 
\begin{array}{ll}
= & -\dsum\limits_k\dint\nolimits_{C_\varepsilon ^k}\sigma _{ij,j}\left(
w_\varepsilon ^{mk}\right) \left( u_\varepsilon -\left( u^n\right)
_\varepsilon ^o\right) _i\left( \left( v^n\right) _m-\left( u^n\right)
_m\right) \varphi _\varepsilon ^kdx \\ 
& -\dsum\limits_k\dint\nolimits_{C_\varepsilon ^k}\sigma _{ij}\left(
w_\varepsilon ^{mk}\right) \left( u_\varepsilon -\left( u^n\right)
_\varepsilon ^o\right) _i \dfrac{\partial \left( \left( \left( v^n\right)
_m-\left( u^n\right) _m\right) \varphi _\varepsilon ^k\right) }{\partial x_j}%
dx \\ 
& +\dsum\limits_k\dint\nolimits_{\partial T_\varepsilon ^k}\sigma
_{ij}\left( w_\varepsilon ^{mk}\right) n_j\left( u_\varepsilon -\left(
u^n\right) _\varepsilon ^o\right) _i\left( \left( v^n\right) _m-\left(
u^n\right) _m\right) dx.%
\end{array}%
\end{array}
\]

Using the estimates of Lemma \ref{2.2}, we prove that the second term above
converges to 0. Using the properties of $w_{\varepsilon }^{mk}$, the first
term above is equal to 0.\ Then, the properties of $w_{\varepsilon }^{mk}$
and the convergence of $\left( u_{\varepsilon }-\left( u^{n}\right)
_{\varepsilon }^{o}\right) _{\varepsilon }$ to $u-u^{n}$ in the weak
topology of $H^{1}\left( \Omega ,\mathbf{R}^{3}\right) $ imply

\[
\begin{array}{r}
\underset{\varepsilon \rightarrow 0}{\lim }\dsum\limits_k\dint\nolimits_{C_%
\varepsilon ^k}\sigma _{ij}\left( w_\varepsilon ^{mk}\right) e_{ij}\left(
u_\varepsilon -\left( u^n\right) _\varepsilon ^o\right) \left( \left(
v^n\right) _m-\left( u^n\right) _m\right) \varphi _\varepsilon ^kdx\qquad \\ 
=2\pi \gamma \dint_\Omega \left( v^n-u^n\right) ^tA\left( u-u^n\right) dx.%
\end{array}
\]

We let $n$ increase to $+\infty $ and get

\[
\underset{\varepsilon \rightarrow 0}{\lim \inf }\dsum\limits_k\dint%
\nolimits_{C_\varepsilon ^k}\sigma _{ij}\left( w_\varepsilon ^{mk}\right)
e_{ij}\left( u_\varepsilon -\left( u^n\right) _\varepsilon ^o\right) \left(
\left( v^n\right) _m-\left( u^n\right) _m\right) \varphi _\varepsilon
^kdx\geq 0, 
\]
which implies, using the computations of the first assertion

\begin{equation}  \label{calc6}
\underset{\varepsilon \rightarrow 0}{\lim \inf }\dint\nolimits_{C_%
\varepsilon }\sigma _{ij}\left( u_\varepsilon \right) e_{ij}\left(
u_\varepsilon \right) dx\geq 2\pi \gamma \dint_\Omega \left( v-u\right)
^tA\left( v-u\right) dx.
\end{equation}

We finally observe that for the third term of $F^\varepsilon \left(
u_\varepsilon \right) $, one has 
\[
\dint\nolimits_{T_\varepsilon }\sigma _{ij}^\varepsilon \left( u_\varepsilon
\right) e_{ij}\left( u_\varepsilon \right) dx\geq \mu ^\varepsilon \dfrac{%
2\mu ^\varepsilon +3\lambda ^\varepsilon }{\mu ^\varepsilon +\lambda
^\varepsilon }\dint\nolimits_{T_\varepsilon }\left( e_{33}\left(
u_\varepsilon \right) \right) ^2dx. 
\]

Indeed, one can easily verify that for every $x$, $y$, $z$ in $\mathbf{R}$,
one has

\[
\lambda ^\varepsilon \left( x+y+z\right) ^2+2\mu ^\varepsilon \left(
x^2+y^2+z^2\right) \geq \mu ^\varepsilon \dfrac{2\mu ^\varepsilon +3\lambda
^\varepsilon }{\mu ^\varepsilon +\lambda ^\varepsilon }z^2. 
\]

We then use the computations given in Lemma \ref{3.5}, which imply, because $%
\mu _{o}$ and $\lambda _{o}$ are finite 
\begin{equation}
\underset{\varepsilon \rightarrow 0}{\lim \inf }\dint\nolimits_{T_{%
\varepsilon }}\sigma _{ij}^{\varepsilon }\left( u_{\varepsilon }\right)
e_{ij}\left( u_{\varepsilon }\right) dx\geq \pi E_{o}\dint_{\Omega }\left(
e_{33}\left( v\right) \right) ^{2}dx.  \label{calc7}
\end{equation}

One deduces from (\ref{calc4})-(\ref{calc7}) 
\[
\begin{array}{r}
\underset{\varepsilon \rightarrow 0}{\lim \inf }F^{\varepsilon }\left(
u_{\varepsilon }\right) \geq \dint_{\Omega }\sigma _{ij}\left( u\right)
e_{ij}\left( u\right) dx+2\pi \gamma \dint_{\Omega }\left( v-u\right)
^{t}A\left( v-u\right) dx\qquad  \\ 
+\pi E_{o}\dint_{\Omega }\left( e_{33}\left( v\right) \right) ^{2}dx,%
\end{array}%
\]%
which concludes the proof.\qquad $\square $

\subsection{Other situations}

The other situations given by different values of the parameters $\gamma $
or $\lambda _o$ or $\mu _o$ are summarized in the

\begin{proposition}
\label{3.6}

\begin{enumerate}
\item If \ $\lambda _{o}$ and $\mu _{o}$ are equal to $0,$ then $\left(
u^{\varepsilon }\right) _{\varepsilon }$ converges in the topology $\tau $
to the solution $\left( u_{o}^{o},v_{o}^{o}\right) $ of the minimization
problem associated to the functional $F_{o}^{o}$ defined in a similar way
than (\ref{Fo}), but with $\lambda _{o}=\mu _{o}=0$.

\item If $\gamma $ is equal to $+\infty $, one obtains $u^{o\infty
}=v^{o\infty }$ in $\Omega $ and $F^{o\infty }$ only depends on $u$ 
\[
F^{o\infty }\left( u\right) =\dint_{\Omega }\sigma _{ij}\left( u\right)
e_{ij}\left( u\right) dx+\pi E_{o}\dint_{\Omega }\left( e_{33}\left(
u\right) \right) ^{2}dx.
\]
\end{enumerate}
\end{proposition}

\begin{proof}
\textit{1}. This case corresponds to a situation where the Lam\'{e}
coefficients $\lambda ^{\varepsilon }$ and $\mu ^{\varepsilon }$ of the
reinforcing material are smaller than the critical ones given in (\ref%
{lambdas}), that is given by%
\[
\lambda _{c}^{\varepsilon }=\dfrac{c\varepsilon ^{2}}{\left( r_{\varepsilon
}\right) ^{2}}\text{, }\mu _{c}^{\varepsilon }=\dfrac{c\varepsilon ^{2}}{%
\left( r_{\varepsilon }\right) ^{2}},
\]%
for every positive and small $c$, but preserving the critical radius $%
r_{\varepsilon }$ of the fibers given through $\gamma $. Let $%
F_{c}^{\varepsilon }$ be the functional defined in (\ref{Feps}) but with
these critical Lam\'{e} coefficients. Thanks to the property of the
epi-convergence, we get, for every $\left( u,v\right) $ in $H_{\Gamma
_{1}}^{1}\left( \Omega ,\mathbf{R}^{3}\right) \times V$%
\[
\begin{array}{r}
F_{o}^{o}\left( u,v\right) \leq F_{c}^{o}\left( u,v\right) =\dint_{\Omega
}\sigma _{ij}\left( u\right) e_{ij}\left( u\right) dx+\pi
cE_{o}\dint_{\Omega }\left( e_{33}\left( v\right) \right) ^{2}dx\qquad  \\ 
+2\pi \gamma \dint_{\Omega }\left( v-u\right) ^{t}A\left( v-u\right) dx.%
\end{array}%
\]

This inequality being true for every positive $c$, we get, letting $c$ go to
0%
\[
F_{o}^{o}\left( u,v\right) \leq \dint_{\Omega }\sigma _{ij}\left( u\right)
e_{ij}\left( u\right) dx+2\pi \gamma \dint_{\Omega }\left( v-u\right)
^{t}A\left( v-u\right) dx.
\]

In order to establish the reverse inequality, we observe that, for every
sequence $\left( u_{\varepsilon }\right) _{\varepsilon }$ converging to $%
\left( u,v\right) $ in the above-defined topology $\tau $, one has%
\[
F^{\varepsilon }\left( u_{\varepsilon }\right) \geq \dint_{\Omega \backslash 
\overline{B_{\varepsilon }}}\sigma _{ij}\left( u_{\varepsilon }\right)
e_{ij}\left( u_{\varepsilon }\right) dx+\dint_{C_{\varepsilon }}\sigma
_{ij}\left( u_{\varepsilon }\right) e_{ij}\left( u_{\varepsilon }\right) dx,
\]%
thus omitting the integral involving the fibers $T_{\varepsilon }$. We then
adapt the proof of the second assertion in the Theorem \ref{3.1} in order to
conclude

\noindent \textit{2}. We again observe that this situation corresponds to a
case where the Lam\'{e} coefficients of the reinforcing material are still
given by (\ref{lambdas}) but where the radius of the fibers is larger than
the critical one, that is $r_{\varepsilon }\geq \exp \left( -1/C\varepsilon
^{2}\right) $, for every positive $C$. The functional $F^{\varepsilon }$ is
thus larger than the functional $F^{\varepsilon C}$ given by (\ref{Feps}),
but with the radius $\exp \left( -1/C\varepsilon ^{2}\right) $. The
comparison principle implies that for every $\left( u,v\right) $ in $%
H_{\Gamma _{1}}^{1}\left( \Omega ,\mathbf{R}^{3}\right) \times V$%
\[
\begin{array}{r}
F^{o\infty }\left( u,v\right) \geq F^{oC}\left( u,v\right) =\dint_{\Omega
}\sigma _{ij}\left( u\right) e_{ij}\left( u\right) dx+\pi E_{o}\dint_{\Omega
}\left( e_{33}\left( v\right) \right) ^{2}dx\qquad  \\ 
+2\pi \gamma C\dint_{\Omega }\left( v-u\right) ^{t}A\left( v-u\right) dx.%
\end{array}%
\]

Letting $C$ increase to $+\infty $, we observe that $F^{o\infty }\left(
u,v\right) $ is finite if and only if the integral $\int_{\Omega }\left(
v-u\right) ^{t}A\left( v-u\right) dx=0$, which implies : $u=v$, in $\Omega $%
. The reverse inequality is still obtained adapting the proof of Theorem \ref%
{3.1} (first part) but with $v=u$.\qquad $\square $
\end{proof}

Let us now examine the special case when $\lambda _{o}=\mu _{o}=+\infty $.
As a special subcase, \cite{Pid-Sep} have considered the case when $\gamma
=+\infty $ and

\begin{equation}  \label{coef1}
\dfrac{\lambda ^\varepsilon \left( r_\varepsilon \right) ^4}{\varepsilon ^2}%
\underset{\varepsilon \rightarrow 0}{\rightarrow }\lambda _1\text{, }\dfrac{%
\mu ^\varepsilon \left( r_\varepsilon \right) ^4}{\varepsilon ^2}\underset{%
\varepsilon \rightarrow 0}{\rightarrow }\mu _1\text{,}
\end{equation}

\noindent with positive and finite $\lambda _1$ and $\mu _1$. We now adapt
their result considering

\begin{proposition}
\label{3.7}Suppose that the above hypothesis (\ref{coef1}) holds true and $%
\gamma $ belongs to $\left] 0,+\infty \right] $. Then, the sequence $\left(
u^{\varepsilon }\right) _{\varepsilon }$ converges in the topology $\tau $,
to the solution $\left( u^{1},v^{1}\right) $ of 
\[
\underset{H_{\Gamma _{1}}^{1}\left( \Omega ,\mathbf{R}^{3}\right) \times
V^{\prime }}{\min }\left( 
\begin{array}{r}
\dint_{\Omega }\sigma _{ij}\left( u\right) e_{ij}\left( u\right) dx+2\pi
\gamma \dint_{\Omega }\left( v-u\right) ^{t}A\left( v-u\right) dx\qquad  \\ 
+\dfrac{\pi E_{1}}{4}\dint_{\Omega }\left( \left( \dfrac{\partial ^{2}v_{1}}{%
\partial x_{3}^{2}}\right) ^{2}+\left( \dfrac{\partial ^{2}v_{2}}{\partial
x_{3}^{2}}\right) ^{2}\right) dx%
\end{array}%
\right) ,
\]%
with $E_{1}=\mu _{1}\left( 3\lambda _{1}+2\mu _{1}\right) /\left( \lambda
_{1}+\mu _{1}\right) $ and 
\[
V^{\prime }=\left\{ v_{\alpha }\in L^{2}\left( \omega ,H^{2}\left(
0,L\right) \right) \mid v_{\mid \Gamma _{1}}=0\text{, }v_{3}=0\right\} .
\]
\end{proposition}

\begin{proof}
We proceed in a similar way to \cite{Pid-Sep}. Indeed, we first follow their
method in order to prove the following estimates%
\[
\dfrac{1}{\left\vert T_{\varepsilon }\right\vert }\dint_{T_{\varepsilon
}}\left\vert u^{\varepsilon }\right\vert dx<C\text{, }\dfrac{1}{\left\vert
T_{\varepsilon }\right\vert }\dint_{T_{\varepsilon }}\left\vert
u^{\varepsilon }\right\vert ^{2}dx<C\text{, }\dfrac{1}{\left( r_{\varepsilon
}\right) ^{2}\left\vert T_{\varepsilon }\right\vert }\dint_{T_{\varepsilon
}}\left\vert e_{ij}\left( u^{\varepsilon }\right) \right\vert ^{2}dx<C,
\]%
where $C$ is independant of $\varepsilon $. For every smooth $v$ in $%
C^{2}\left( \overline{\Omega },\mathbf{R}^{3}\right) \cap V^{\prime }$, we
set%
\[
\left\{ 
\begin{array}{l}
\left( \mathcal{R}_{\varepsilon 1}\left( v\right) \right) _{1}\left(
x_{1},x_{2},x_{3}\right) =v_{1}\left( k_{1}\varepsilon ,k_{2}\varepsilon
,x_{3}\right)  \\ 
\qquad -\dfrac{\lambda ^{\varepsilon }}{2\left( \mu ^{\varepsilon }+\lambda
^{\varepsilon }\right) }\dfrac{\left( x_{1}-k_{1}\varepsilon \right)
^{2}-\left( x_{2}-k_{2}\varepsilon \right) ^{2}}{2}\dfrac{\partial ^{2}v_{1}%
}{\partial x_{3}^{2}}\left( k_{1}\varepsilon ,k_{2}\varepsilon ,x_{3}\right) 
\\ 
\qquad -\dfrac{\lambda ^{\varepsilon }}{2\left( \mu ^{\varepsilon }+\lambda
^{\varepsilon }\right) }\dfrac{\left( x_{1}-k_{1}\varepsilon \right) \left(
x_{2}-k_{2}\varepsilon \right) }{2}\dfrac{\partial ^{2}v_{2}}{\partial
x_{3}^{2}}\left( k_{1}\varepsilon ,k_{2}\varepsilon ,x_{3}\right)  \\ 
\left( \mathcal{R}_{\varepsilon 1}\left( v\right) \right) _{2}\left(
x_{1},x_{2},x_{3}\right) =v_{2}\left( k_{1}\varepsilon ,k_{2}\varepsilon
,x_{3}\right)  \\ 
\qquad -\dfrac{\lambda ^{\varepsilon }}{2\left( \mu ^{\varepsilon }+\lambda
^{\varepsilon }\right) }\dfrac{\left( x_{1}-k_{1}\varepsilon \right)
^{2}-\left( x_{2}-k_{2}\varepsilon \right) ^{2}}{2}\dfrac{\partial ^{2}v_{2}%
}{\partial x_{3}^{2}}\left( k_{1}\varepsilon ,k_{2}\varepsilon ,x_{3}\right) 
\\ 
\qquad -\dfrac{\lambda ^{\varepsilon }}{2\left( \mu ^{\varepsilon }+\lambda
^{\varepsilon }\right) }\dfrac{\left( x_{1}-k_{1}\varepsilon \right) \left(
x_{2}-k_{2}\varepsilon \right) }{2}\dfrac{\partial ^{2}v_{1}}{\partial
x_{3}^{2}}\left( k_{1}\varepsilon ,k_{2}\varepsilon ,x_{3}\right)  \\ 
\left( \mathcal{R}_{\varepsilon 1}\left( v\right) \right) _{3}\left(
x_{1},x_{2},x_{3}\right) =-\left( x_{1}-k_{1}\varepsilon \right) \dfrac{%
\partial v_{1}}{\partial x_{3}}\left( k_{1}\varepsilon ,k_{2}\varepsilon
,x_{3}\right)  \\ 
\qquad -\left( x_{2}-k_{2}\varepsilon \right) \dfrac{\partial v_{2}}{%
\partial x_{3}}\left( k_{1}\varepsilon ,k_{2}\varepsilon ,x_{3}\right) .%
\end{array}%
\right. 
\]

The verification of the first assertion of the epi-convergence is obtained
computing the energy of the test-function associated to this $\mathcal{R}%
_{\varepsilon 1}\left( v\right) $. The verification of the second assertion
follows the same lines as in Theorem \ref{3.1}.\qquad $\square $
\end{proof}

\begin{remark}
\label{3.8}The extra term occuring in the energy functional described in
Proposition \ref{3.7} corresponds to the flexion of the fibers.
\end{remark}

\begin{remark}
\label{3.9}In the case $\gamma =0$, one can still prove that $\left(
\int_{T_{\varepsilon }}\left\vert \left( u_{\varepsilon }\right)
_{3}\right\vert dx/\left\vert T_{\varepsilon }\right\vert \right)
_{\varepsilon }$ is bounded, writing : $u_{\varepsilon }\left( s\right)
=\int_{0}^{s}\partial \left( u_{\varepsilon }\right) _{3}/\partial x_{3}dt$
and using some trivial arguments.\ Thus Lemma \ref{3.5} still implies the
existence of $e_{33}\left( v\right) $ in $L^{2}\left( \Omega \right) $, with 
$v_{3}=0$\ on $\Gamma _{1}$. We conjecture that the limit functional is 
\[
F^{oo}\left( u,v\right) =\dint_{\Omega }\sigma _{ij}\left( u\right)
e_{ij}\left( u\right) dx+\pi E_{o}\dint_{\Omega }\left( e_{33}\left(
v\right) \right) ^{2}dx.
\]
\end{remark}

\section{Further extensions}

\subsection{The case of a almost non-periodic distribution of fibers}

Let $\widetilde{\omega }$ be some open subset of $\mathbf{R}^2$ and $\theta $
be a $C^1$-diffeomorphism from $\widetilde{\omega }$ to $\omega $. We define
the following almost non-periodic distribution of non-homogeneous fibers as
follows. The fibers are defined as

\[
T_\varepsilon ^k=\left\{ \left( x_1,x_2,x_3\right) \mid \left( x_1-\theta
_1\left( k_1\varepsilon ,k_2\varepsilon \right) \right) ^2+\left( x_2-\theta
_2\left( k_1\varepsilon ,k_2\varepsilon \right) \right) ^2<\left(
r_\varepsilon \right) ^2\text{, }x_3\in \left] 0,L\right[ \right\} . 
\]

Replacing $\left( k_{1}\varepsilon ,k_{2}\varepsilon \right) $ by $\theta
\left( k_{1}\varepsilon ,k_{2}\varepsilon \right) $ in the local
test-functions and adapting the proof of Theorem \ref{3.1}, one can prove

\begin{theorem}
\label{4.1}Suppose that $\gamma $ is positive and finite and the
nonhomogeneous material filling in the fibers satisfies the usual conditions
of symmetry, uniform ellipticity and continuity and 
\[
\underset{x_{3}\in \lbrack 0,L],\varepsilon >0}{\sup }\left\vert \dfrac{%
\left( r_{\varepsilon }\right) ^{2}}{\varepsilon ^{2}}a_{ijkl}^{\varepsilon
}\left( x_{3}\right) \right\vert <+\infty \text{, }\dfrac{\left(
r_{\varepsilon }\right) ^{2}}{\varepsilon ^{2}}a_{ijkl}^{\varepsilon }\left(
x_{3}\right) \underset{\varepsilon \rightarrow 0}{\rightarrow }%
a_{ijkl}^{o}\left( x_{3}\right) \text{, }a.e.\text{ in }\Omega .
\]%
Then, the sequence $\left( F^{\varepsilon }\right) _{\varepsilon }$
epi-converges in the topology $\tau $ to the functional $F^{o}$ defined on $%
H^{1}\left( \Omega ,\mathbf{R}^{3}\right) \times L^{1}\left( \Omega ,\mathbf{%
R}^{3}\right) $ by: 
\[
F^{o}\left( u,v\right) =\left\{ 
\begin{array}{l}
\dint_{\Omega }\sigma _{ij}\left( u\right) e_{ij}\left( u\right) dx+2\pi
\gamma \dint_{\Omega }\left( v-u\right) ^{t}A\left( v-u\right) \left\vert
\nabla \theta ^{-1}\right\vert \left( x_{1},x_{2}\right) dx\  \\ 
\qquad +\pi \dint_{\Omega }E^{o}\left( x_{3}\right) e_{33}\left( v\right)
e_{33}\left( v\right) \left\vert \nabla \theta ^{-1}\right\vert \left(
x_{1},x_{2}\right) dx, \\ 
\hfill \text{if }\left( u,v\right) \in H_{\Gamma _{1}}^{1}\left( \Omega ,%
\mathbf{R}^{3}\right) \times V \\ 
+\infty \hfill \text{otherwise,}%
\end{array}%
\right. 
\]%
where $E^{o}\left( x_{3}\right) $ is Young's modulus associated to $%
a_{ijkl}^{o}\left( x_{3}\right) $.
\end{theorem}

\subsection{The case of tranverse fibers}

Let us assume in this paragraph that $\omega $ is the disk centred at the
origin and of radius $R>0$ of $\mathbf{R}^2$. Choose any $R^{*}$ in $\left]
0,R\right] $ and positive $\varepsilon $ and $r_\varepsilon $ such that : $%
0<2r_\varepsilon <\varepsilon <1$. For every $k$ in $\mathbf{Z}$, we
introduce the torus $T_\varepsilon ^k$ defined as 
\[
T_\varepsilon ^k=\left\{ \left( x_1,x_2,x_3\right) \in \mathbf{R}^3\mid
\left( R^{*}-\sqrt{\left( x_1\right) ^2+\left( x_2\right) ^2}\right)
^2+\left( x_3-k\varepsilon \right) ^2<\left( r_\varepsilon \right)
^2\right\} . 
\]

$T_\varepsilon $ denotes the union $\dbigcup $ $_{k=-n(\varepsilon
)}^{k=n(\varepsilon )}T_\varepsilon ^k$ of the tori $T_\varepsilon ^k$ $%
\varepsilon $-periodically distributed along the surface : $\Sigma
_{R^{*}}=\left\{ \left( x_1\right) ^2+\left( x_2\right) ^2=\left(
R^{*}\right) ^2\text{, }x_3\in \left] 0,L\right[ \right\} $ and contained in 
$\Omega =\omega \times \left] 0,L\right[ $. We suppose that $\overline{%
T_\varepsilon }\cap \Gamma _1$ and $\overline{T_\varepsilon }\cap \Gamma _2$
are empty. The number $n\left( \varepsilon \right) $ of such tori contained
in $\Omega $ is equivalent to $L/\varepsilon $.

We define the topology $\tau ^{*}$ as

\[
\begin{array}{c}
u_\varepsilon \overset{\tau ^{*}}{\underset{\varepsilon \rightarrow 0}{%
\rightharpoonup }}\left( u,v\right) \Leftrightarrow u_\varepsilon \overset{w%
\text{-}H^1\left( \Omega ,\mathbf{R}^3\right) }{\underset{\varepsilon
\rightarrow 0}{\rightharpoonup }}u \\ 
\text{and : }\forall \varphi \in C_c^o\left( \mathbf{R}^3\right)
:\dint_{\Sigma _{R^{*}}}\left( R^{\varepsilon ^{*}}\left( u_\varepsilon
\right) \varphi \right) _{\mid \Sigma _{R^{*}}}d\sigma \underset{\varepsilon
\rightarrow 0}{\rightarrow }\dint_{\Sigma _{R^{*}}}\left( v\varphi \right)
_{\mid \Sigma _{R^{*}}}d\sigma ,%
\end{array}
\]

\noindent where $R^{\varepsilon ^{*}}$ is defined by : $R^{\varepsilon
^{*}}\left( u\right) =\left| \Sigma _{R^{*}}\right| u\mathbf{1}%
_{T_\varepsilon }/\left| T_\varepsilon \right| $. We introduce the space

\[
V^{*}=\left\{ 
\begin{array}{r}
v=\left( v_r,v_\theta ,v_{x_3}\right) :\left[ 0,2\pi \right] \times \left]
0,L\right[ \rightarrow \mathbf{R}^3\mid v_\alpha \in L^2\left( \left] 0,2\pi %
\right[ \times \left] 0,L\right[ \right) ,\quad \\ 
\quad v_\alpha \left( 0,.\right) =v_\alpha \left( 2\pi ,.\right) \text{, }%
\alpha =r,\theta ,x_3\text{, }\dfrac{\partial v_\theta }{\partial \theta }%
+v_r\in L^2\left( \left] 0,2\pi \right[ \times \left] 0,L\right[ \right) .%
\end{array}
\right\} 
\]

\begin{center}
\FRAME{fhFU}{3.0234in}{1.9761in}{0pt}{\Qcb{The cylinder $\Omega $ and the
tori $T_{\protect\varepsilon }^{k}$.}}{}{jar5.gif}{\special{language
"Scientific Word";type "GRAPHIC";maintain-aspect-ratio TRUE;display
"USEDEF";valid_file "F";width 3.0234in;height 1.9761in;depth
0pt;original-width 2.9793in;original-height 1.9372in;cropleft "0";croptop
"1";cropright "1";cropbottom "0";filename 'JAR5.gif';file-properties
"XNPEU";}}
\end{center}

Following similar arguments to the ones presented in the previous parts, we
prove

\begin{theorem}
\label{4.2}Suppose that $\gamma ^{\ast }=\lim_{\varepsilon \rightarrow
0}\left( -1/\left( \varepsilon \ln r_{\varepsilon }\right) \right) $ is
finite, $\lambda _{o}^{\ast }$ and $\mu _{o}^{\ast }$ are finite and $\mu
_{o}^{\ast }$ is positive, with:%
\[
\lambda _{o}^{\ast }=\ \underset{\varepsilon \rightarrow 0}{\lim }\dfrac{%
\lambda ^{\varepsilon }\left( r_{\varepsilon }\right) ^{2}}{\varepsilon }%
\text{, }\mu _{o}^{\ast }=\text{ }\underset{\varepsilon \rightarrow 0}{\lim }%
\dfrac{\mu ^{\varepsilon }\left( r_{\varepsilon }\right) ^{2}}{\varepsilon }.
\]%
Then, the sequence $\left( F^{\varepsilon }\right) _{\varepsilon }$
epi-converges in the topology $\tau ^{\ast }$ to the functional $F^{o^{\ast
}}$ defined on $H^{1}\left( \Omega ,\mathbf{R}^{3}\right) \times L^{1}\left(
\Omega ,\mathbf{R}^{3}\right) $ by: 
\[
F^{o^{\ast }}\left( u,v\right) =\left\{ 
\begin{array}{l}
\dint_{\Omega }\sigma _{ij}\left( u\right) e_{ij}\left( u\right) dx+\pi
E_{o}^{\ast }\dint_{0}^{2\pi }\dint_{0}^{L}\left( \dfrac{\partial v_{\theta }%
}{\partial \theta }+v_{r}\right) ^{2}\left( R^{\ast },\theta ,x_{3}\right)
d\theta dx_{3} \\ 
\quad +2\pi \gamma ^{\ast }R^{\ast }\dint_{0}^{2\pi }\dint_{0}^{L}\left(
v-u_{\mid \Sigma _{R^{\ast }}}\right) ^{t}A\left( v-u_{\mid \Sigma _{R^{\ast
}}}\right) \left( R^{\ast },\theta ,x_{3}\right) d\theta dx_{3}, \\ 
\hfill \text{if }\left( u,v\right) \in H_{\Gamma _{1}}^{1}\left( \Omega ,%
\mathbf{R}^{3}\right) \times V^{\ast } \\ 
+\infty \hfill \text{otherwise,}%
\end{array}%
\right. 
\]%
with $A$ as in Theorem \ref{3.1} and $E_{o}^{\ast }=\mu _{o}^{\ast }\left(
3\lambda _{o}^{\ast }+2\mu _{o}^{\ast }\right) /\left( \lambda _{o}^{\ast
}+\mu _{o}^{\ast }\right) $.
\end{theorem}

\textbf{ACKNOWLEDGEMENT}

We thank the referee for the useful comments on this work.

\end{document}